\documentclass[11pt]{amsart}
\usepackage{amsmath,amssymb,amsthm,latexsym,cite,cancel}
\usepackage[small]{caption}
\usepackage{graphicx,wasysym,overpic,tikz,color}
\usepackage{subfigure,color}
\usepackage{cite}
\usepackage[colorlinks=true,urlcolor=blue,
citecolor=red,linkcolor=blue,linktocpage,pdfpagelabels,
bookmarksnumbered,bookmarksopen]{hyperref}
\usepackage[italian,english]{babel}
\usepackage{units}
\usepackage{enumitem}
\usepackage[left=2.1cm,right=2.1cm,top=2.71cm,bottom=2.71cm]{geometry}
\usepackage[hyperpageref]{backref}
\usepackage[colorinlistoftodos]{todonotes}
\newtheorem{theorem}{Theorem}[section]

\newtheorem{proposition}[theorem]{Proposition}
\newtheorem{lemma}[theorem]{Lemma}

\newtheorem{corollary}[theorem]{Corollary}

\numberwithin{equation}{section}

\title[Critical magnetic Schr\"odinger equation in $\mathbb{R}^2$]{Multiplicity and concentration results for a magnetic Schr\"{o}dinger equation with exponential critical growth in $\mathbb{R}^{2}$}

\author[P. d'Avenia]{Pietro d'Avenia}

\address[P. d'Avenia]{\newline\indent
	Dipartimento di Meccanica, Matematica e Management
	\newline\indent
	Politecnico di Bari
	\newline\indent
	Via Orabona 4,  70125  Bari, Italy}
\email{\href{mailto:pietro.davenia@poliba.it}{pietro.davenia@poliba.it}}

\author[C. Ji]{Chao Ji}

\address[C. Ji]{\newline\indent
	Department of Mathematics
	\newline\indent
	East China University of Science and Technology
	\newline\indent
	Shanghai 200237, PR China }
\email{\href{mailto:jichao@ecust.edu.cn}{jichao@ecust.edu.cn}}

\thanks{P. d'Avenia is supported by PRIN project 2017JPCAPN {\em Qualitative and quantitative aspects of nonlinear PDEs}. C. Ji is partially supported by Shanghai Natural Science Foundation (18ZR1409100).}

\begin{document}

\maketitle

\begin{abstract}
In this paper  we study the  following nonlinear Schr\"{o}dinger equation with magnetic field
\begin{align*}
\Big(\frac{\varepsilon}{i}\nabla-A(x)\Big)^{2}u+V(x)u=f(| u|^{2})u,\quad x\in\mathbb{R}^{2},
\end{align*}
where  $\varepsilon>0$ is a parameter,  $V:\mathbb{R}^{2}\rightarrow \mathbb{R}$ and $A: \mathbb{R}^{2}\rightarrow \mathbb{R}^{2}$ are continuous
potentials and $f:\mathbb{R}\rightarrow \mathbb{R}$ has exponential critical
growth. Under a local assumption on the potential $V$,  by variational methods, penalization technique, and Ljusternick-Schnirelmann theory, we prove multiplicity and concentration of solutions for $\varepsilon$ small.
\end{abstract}

\maketitle
\begin{center}
	\begin{minipage}{11cm}
		\tableofcontents
	\end{minipage}
\end{center}

\section{Introduction and main results}

In this paper, we are concerned with multiplicity and concentration results for the following nonlinear  magnetic Schr\"{o}dinger equation
\begin{equation}
\label{1.1}
\Big(\frac{\varepsilon}{i}\nabla-A(x)\Big)^{2}u+V(x)u=f(|u|^{2})u
\quad
\hbox{in }\mathbb{R}^2,
\end{equation}
where $u\in H^{1}(\mathbb{R}^{2}, \mathbb{C})$, $\varepsilon>0$ is a parameter, $V:\mathbb{R}^{2}\rightarrow \mathbb{R}$ is a continuous function,  $f:\mathbb{R}\rightarrow \mathbb{R}$,
and the magnetic potential $A: \mathbb{R}^{2}\rightarrow\mathbb{R}^{2}$ is H\"{o}lder continuous with exponent $\alpha\in(0, 1]$.

Equation \eqref{1.1} arises when one looks for standing wave solutions $\psi(x, t):=e^{-iEt/\hbar}u(x)$, with $E\in \mathbb{R}$, of
\begin{equation*}
i\hbar \frac{\partial \psi}{\partial t}=\Big(\frac{\hbar}{i}\nabla-A(x)\Big)^{2}\psi+U(x)\psi-f(|\psi|^{2})\psi
\quad
\hbox{in }\mathbb{R}^2\times\mathbb{R}.
\end{equation*}
From a physical point of view, the existence of such solutions and the study of their shape  in the semiclassical limit, namely, as $\hbar\rightarrow 0^{+}$, or, equivalently, as $\varepsilon\to 0^+$ in \eqref{1.1}, is of the greatest importance, since the transition from Quantum Mechanics to Classical Mechanics can be formally performed by sending the Planck constant $\hbar$ to zero.

For equation \eqref{1.1}, there is a vast literature concerning the existence and multiplicity of bound state solutions, in particular for the case with $A\equiv 0$.
The first result in this direction was given by Floer and Weinstein in \cite{rFW}, where the case $N=1$ and $f=i_\mathbb{R}$ is considered.
Later, many authors generalized this result to larger values of $N$, using different methods. In \cite{rDF}, del Pino and Felmer studied existence and concentration of the solutions for the following problem
\[
\begin{cases}
-\varepsilon^{2}\Delta u+V(x)u=f(u) \hbox{ in }\Omega,\\
u=0 \text{ on } \partial \Omega, \;
u>0 \hbox{ in }\Omega,
\end{cases}
\]
where $\Omega$ is a possibly unbounded domain in $\mathbb{R}^{N}$, $N\geq 3$, the potential $V$ is locally H\"older continuous, bounded from below away from zero, there exists a bounded open set $\Lambda\subset \Omega$ such that
\[
\inf_{x\in \Lambda}V(x)<\min_{x\in \partial \Lambda}V(x),
\]
and the nonlinearity $f$ satisfies some subcritical growth conditions.
For further results about existence, multiplicity and qualitative properties of semiclassical states with various types of concentration behaviors, which have been established under various assumptions on the potential $V$ and on the nonlinearity $f$, see \cite{rACTY,rAF,rAF1,rAFN,rAS,rABC,rAMS,rBC,rBT1,rBT2,rDPR,rO1, rO2, rW1,rZ} the references therein (see also \cite{rADM,rAM,rHZ} for the fractional case).

On the other hand, also the  magnetic nonlinear Schr\"{o}dinger equation \eqref{1.1} has been extensively investigated  by many authors applying suitable variational
and topological methods (see \cite{rAFF,rAZ,rBCS,rBF,rCi,rCS1,rEL,rLS} and references therein). It is well known that the first result involving the magnetic field was obtained by Esteban and Lions \cite{rEL}. They used the concentration-compactness principle and minimization arguments to obtain solutions for $\varepsilon>0$ fixed and $N=2, 3$.
In particular, due to our scope, we want to mention \cite{rAFF} where the authors use the penalization method and Ljusternik-Schnirelmann category theory for subcritical nonlinearities and \cite{rBF} where the existence of a complex solution in presence of a nonlinearity with exponential critical growth in $\mathbb{R}^{2}$ is proved, and the recent contribution \cite{rAMV} where a multiplicity result for a nonlinear fractional magnetic Schr\"{o}dinger equation
with exponential critical growth in the one-dimensional case is given.

In this paper, motivated by \cite{rAFF, rDF}, we prove  multiplicity and concentration of nontrivial solutions for problem \eqref{1.1}, combining some assumptions on $V$, the penalization technique by del Pino and Felmer \cite{rDF} and the Ljusternik-Schnirelmann theory.

Assume that $V$ verifies the following properties:
\begin{enumerate}[label=($V_\arabic{*}$), ref=$V_\arabic{*}$]
	\item \label{V1}there exists $V_{0}>0$ such that $V(x)\geq V_{0}$ for all $x\in \mathbb{R}^{2}$;
	\item \label{V2}there exists a bounded open set $\Lambda\subset \mathbb{R}^{2}$ such that
	\[
	V_{0}=\min_{x\in \Lambda}V(x)<\min_{x\in \partial \Lambda}V(x).
	\]
\end{enumerate}
Observe that $$M:=\{x\in\Lambda: V(x)=V_{0}\}\neq \emptyset.$$
Moreover, let the nonlinearity $f$ be a $C^{1}$-function satisfying:
\begin{enumerate}[label=($f_\arabic{*}$), ref=$f_\arabic{*}$]
	\item \label{f1}$f(t)=0$ if $t\leq0$;
	\item \label{f2}there holds
	\[
	\lim_{t\rightarrow +\infty} \frac{f(t^{2})t}{e^{\alpha t^{2}}}
	=
	\begin{cases}
	0,& \hbox{for } \alpha> 4\pi,\\
	+\infty,& \hbox{for }  0<\alpha<4\pi;
	\end{cases}
	\]
	\item \label{f3}there is a positive constant $\theta>2$ such that
	\begin{equation*}
	0<\frac{\theta}{2} F(t)\leq tf(t), \quad \quad \forall\,  t>0,
	\quad
	\hbox{where }
	F(t)=\int_{0}^{t}f(s)ds;
	\end{equation*}
	\item \label{f4}there exist two constants $p>2$ and
	\begin{equation*}
	C_{p}>\max\left\{\Big[\beta_{p}\Big(\frac{2\theta}{\theta-2}\Big)\frac{1}{\min\{1, V_{0}\}}\Big]^{(p-2)\diagup 2}, V_0\left(\frac{p-2}{p}\right)^\frac{p-2}{2} S_p^{p/2} \right\}>0
	\end{equation*}
	such that
    \begin{equation*}
	f'(t)\geq \frac{p-2}{2}C_{p}t^{(p-4)/2}\quad \text{for all}\,\, t>0,
	\end{equation*}
	where
	\begin{equation*}
	\beta_{p}=\inf_{u\in \tilde{\mathcal{N}}_{0}}\tilde{I}_{0}(u),
	\quad
	\tilde{I}_{0}(u):=\frac{1}{2}\int_{\mathbb{R}^{2}}(\vert \nabla u\vert^{2}+V_{0}\vert u\vert^{2})dx-\frac{1}{p}\int_{\mathbb{R}^{2}}\vert u\vert^{p}dx,
	\end{equation*}
	\begin{equation*}
	\tilde{\mathcal{N}}_{0}:=\{u\in H^{1}(\mathbb{R}^{2}, \mathbb{R})\backslash\{0\}: \tilde{I}'_{0}(u)[u]=0\},
	\end{equation*}
and $S_p$ is the best constant for the Sobolev inequality
	$$S_p\Big( \int_{\mathbb{R}^2}| u |^p dx\Big)^{2/p} \leq  \int_{\mathbb{R}^2}| (\nabla u |^2 +| u |^2) dx;$$
    \item \label{f5} $f'(t)\leq (e^{4\pi t}-1)$ for any $t\geq0$.
\end{enumerate}
Observe that assumptions (\ref{f4}) and (\ref{f5}) imply that there exist two positive constants $C_1$ and $C_2$ such that
$$
C_1 t^{(p-4)/2}\leq f'(t)\leq C_{2} t,\quad \text{as } t\rightarrow 0^{+}
$$
and so $p> 6$.

Our main result is
\begin{theorem}\label{mt}
Assume that $V$ satisfies (\ref{V1}), (\ref{V2}) and $f$ satisfies (\ref{f1})--(\ref{f5}). Then, for any  $\delta>0$ such that
\begin{equation*}
M_{\delta}:=\{x\in \mathbb{R}^{2}: dist(x, M)<\delta\}\subset \Lambda,
\end{equation*}
there exists $\varepsilon_{\delta}>0$ such that, for any $0<\varepsilon<\varepsilon_{\delta}$, problem \eqref{1.1} has at least  $\text{cat}_{M_{\delta}}(M)$ nontrivial solutions. Moreover, for every sequence $\{\varepsilon_n\}$ such that $ \varepsilon_{n}\rightarrow 0^{+}$ as $n\to+\infty$, if we denote by $u_{\varepsilon_{n}}$ one of these solutions of \eqref{1.1} for $\varepsilon=\varepsilon_{n}$ and $\eta_{\varepsilon_{n}}\in \mathbb{R}^{2}$ the global maximum point of $\vert u_{\varepsilon_{n}}\vert$, then
	\begin{equation*}
	\lim_{\varepsilon_{n}\rightarrow 0^{+}}V(\eta_{\varepsilon_{n}})=V_{0}.
	\end{equation*}
\end{theorem}

It is well known that, when we want to study by variational methods this type of equations  in the whole $\mathbb{R}^{2}$, we meet several difficulties due to the unboundedness of the domain and to the exponential critical growth of the nonlinearity.
Moreover, we only know local information on the potential $V$, and we don't have any condition on $V$ at infinity.
Thus we adapt the penalization technique explored
in \cite{rDF}. It consists in making a suitable modification on the nonlinearity $f$, solving a modified problem  and then check that, for $\varepsilon$ small
enough, the solutions of the modified problem are indeed solutions of the original one. It could be interesting to consider our problem without relying upon condition near $0$.\\
It is worthwhile to remark that in the arguments developed in \cite{rDF}, one of the key points is the existence of estimates involving the $L^{\infty}$-norm of the solutions of the modified problem.
In the the magnetic case,  this kind of estimates are more delicate, due also to the fact that we deal with complex valued functions. For subcritical nonlinearities, Alves et al. in \cite{rAFF} obtained $L^{\infty}$-estimates of the solutions of the modified problem by a different approach, which is based on Moser's iteration method (see \cite{rM}) instead of Kato's inequality. Here the problem we deal with has exponential critical growth in $\mathbb{R}^{2}$, so the method in \cite{rAFF} does not seem fully applicable.

The paper is organized as follows. In Section \ref{Sec2} we introduce the functional setting, give some preliminaries and study the limit problem.
In Section \ref{sec3}, we study the modified
problem. We prove the Palais-Smale condition for the modified functional and provide some tools which are useful to establish a multiplicity result. In Section \ref{sec4}, we show a multiplicity result for he modified
problem. Finally, in Section \ref{Sec5},  we complete the paper with the proof of Thereom \ref{mt}.

\subsection*{Notation}
\begin{itemize}
	\item $C, C_1, C_2, \ldots$ denote positive constants whose exact values are inessential and can change from line to line;
	\item $B_{R}(y)$ denotes the open disk centered at $y\in\mathbb{R}^2$ with radius $R>0$ and $B^{c}_{R}(y)$ denotes the complement of $B_{R}(y)$ in $\mathbb{R}^{2}$;
	\item $\Vert \cdot \Vert$, $\Vert \cdot \Vert_{q}$, and $\Vert\cdot \Vert_{L^{\infty}(\Omega)}$ denote the usual norms of the spaces $H^{1}(\mathbb{R}^{2}, \mathbb{R})$, $L^{q}(\mathbb{R}^{2}, \mathbb{R})$, and $L^{\infty}(\Omega, \mathbb{R})$, respectively, where $\Omega\subset \mathbb{R}^{2}$, and $\Vert \cdot \Vert_{V_{0}}:=(\|\nabla\cdot\|_2+V_0\|\cdot\|_2)^{1/2}$.
\end{itemize}

\section{The variational framework and the limit problem}\label{Sec2}

In this section, we present the functional spaces that we use, we introduce a {\em classical} equivalent version of \eqref{1.1}, we give some useful preliminary remarks, and we study a {\em limit problem} which will be useful for our arguments.

For $u: \mathbb{R}^{2}\rightarrow \mathbb{C}$, let us denote by
\begin{equation*}
\nabla_{A}u:=\Big(\frac{\nabla}{i}-A\Big)u,
\end{equation*}
and
\begin{equation*}
	H_{A}^{1}(\mathbb{R}^{2}, \mathbb{C}):=\{u\in L^{2}(\mathbb{R}^{2}, \mathbb{C}): |\nabla_{A}u|\in L^{2}(\mathbb{R}^{2}, \mathbb{R})\}.
	\end{equation*}
The space $H_{A}^{1}(\mathbb{R}^{2}, \mathbb{C})$ is an Hilbert space endowed with the scalar product
\begin{equation*}
\langle u, v\rangle
:=\operatorname{Re}\int_{\mathbb{R}^{2}}\Big(\nabla_{A}u\overline{\nabla_{A}v}+u\overline{v}\Big)dx,\quad\text{for any } u, v\in H_{A}^{1}(\mathbb{R}^{2}, \mathbb{C}),
\end{equation*}
where $\operatorname{Re}$ and the bar denote the real part of a complex number and the complex conjugation, respectively. Moreover we denote by $\Vert u\Vert_{A}$ the norm induced by this inner product.

On $H_{A}^{1}(\mathbb{R}^{2}, \mathbb{C})$ we will frequently use the following diamagnetic inequality (see e.g. \cite[Theorem~7.21]{rLL})
\begin{equation}
\label{2.1}
\vert \nabla_{A}u(x)\vert\geq \vert \nabla\vert u(x)\vert\vert.
\end{equation}

Moreover, making a simple change of variables, we can see that \eqref{1.1} is equivalent to
\begin{equation}
\label{1.6}
\Big(\frac{1}{i}\nabla-A_{\varepsilon}(x)\Big)^{2}u+V_{\varepsilon}(x)u=f(\vert u\vert^{2})u
\quad
\hbox{in }\mathbb{R}^2,
\end{equation}
where $A_{\varepsilon}(x)=A(\varepsilon x)$ and $V_{\varepsilon}(x)=V(\varepsilon x)$.

Let $H_{\varepsilon}$ be the Hilbert space obtained as the closure of $C_{c}^{\infty}(\mathbb{R}^{2}, \mathbb{C})$ with respect to the scalar product
\begin{align*}
\langle u, v\rangle_{\epsilon}
:=\operatorname{Re}\int_{\mathbb{R}^{2}}\Big(\nabla_{A_{\varepsilon}}u\overline{\nabla_{A_{\varepsilon}}v}+V_{\varepsilon}(x)u\overline{v}\Big)dx
\end{align*}
and let us denote by $\Vert \cdot \Vert_{\varepsilon}$ the norm induced by this inner product.

 The diamagnetic inequality \eqref{2.1} implies that, if $u\in H_{A_{\varepsilon}}^{1}(\mathbb{R}^{2}, \mathbb{C})$, then $|u|\in H^{1}(\mathbb{R}^{2}, \mathbb{R})$  and $\|u\|\leq C \|u\|_\varepsilon$. Therefore, the embedding $H_{\varepsilon}\hookrightarrow L^{r}(\mathbb{R}^{2}, \mathbb{C})$ is continuous for $r\geq 2$ and
the embedding $H_{\varepsilon}\hookrightarrow L^{r}_{\text{loc}}(\mathbb{R}^{2}, \mathbb{C})$ is compact for $r\geq 1$.

About the nonlinearity, we observe that,  by (\ref{f1}) and (\ref{f2}), fixed $q>2$, for any $\zeta>0$ and $\alpha>4\pi$, there exists a constant $C>0$, which depends on $q$, $\alpha$, $\zeta$, such that
	\begin{equation}
	\label{1.2}
	f(t)\leq\zeta+Ct^{(q-2)/2}(e^{\alpha t}-1) \text{ for all } t\geq 0
	\end{equation}
	and, using (\ref{f3}), we have
	\begin{equation}
	\label{1.3}
	F(t)\leq\zeta t+Ct^{q/2}(e^{\alpha t}-1) \text{ for all } t\geq 0.
	\end{equation}
	Moreover, it is easy to see that, by \eqref{1.2} and \eqref{1.3},
	\begin{equation}
	\label{1.4}
	f(t^{2})t^{2}\leq\zeta t^{2}+C\vert t\vert^{q}(e^{\alpha  t^{2}}-1) \text{ for all } t\in\mathbb{R}
	\end{equation}
	and
	\begin{equation}
	\label{1.5}
	F(t^{2})\leq\zeta t^{2}+C\vert t\vert^{q}(e^{\alpha t^{2}}-1)
	\text{ for all } t\in\mathbb{R}.
	\end{equation}

Finally, let us recall the following version of Trudinger-Moser inequality as stated e.g. in \cite[Lemma 1.2]{rACTY}.
\begin{lemma}\label{le24}
	If $\alpha>0$ and $u\in H^{1}(\mathbb{R}^{2}, \mathbb{R})$, then
	\begin{equation*}
	\int_{\mathbb{R}^{2}}(e^{\alpha  u^{2}}-1)dx<+\infty.
	\end{equation*}
	Moreover, if $\Vert \nabla u\Vert_{2}^{2}\leq 1$, $\Vert  u\Vert_{2}\leq M<+\infty$, and $0<\alpha< 4\pi$, then there exists a positive constant $C(M, \alpha)$, which depends only on $M$ and $\alpha$,  such that
	\begin{equation*}
	\int_{\mathbb{R}^{2}}(e^{\alpha u^{2}}-1)dx\leq C(M, \alpha).
	\end{equation*}
\end{lemma}

For our scope, we need also to study the following {\em limit} problem
\begin{align}
\label{3.2}
-\Delta u+V_{0}u=f(u^{2})u, \quad u:\mathbb{R}^{2}\to\mathbb{R},
\end{align}
whose associated $C^1$-functional, defined in $H^{1}(\mathbb{R}^{2}, \mathbb{R})$, is
\begin{equation*}
I_{V_{0}}(u):=\frac{1}{2}\int_{\mathbb{R}^{2}}(\vert \nabla u\vert^{2}+V_{0} u^{2})dx -\frac{1}{2}\int_{\mathbb{R}^{2}}F(u^{2})dx.
\end{equation*}
Let
\begin{equation*}
\mathcal{N}_{V_{0}}:=\{u\in H^{1}(\mathbb{R}^{2}, \mathbb{R})\setminus\{0\}: I'_{V_{0}}(u)[u]=0\}
\end{equation*}
and
\begin{equation*}
c_{V_{0}}:=\inf_{u\in \mathcal{N}_{V_{0}}} I_{V_{0}}(u).
\end{equation*}
By (\ref{f1}) and (\ref{f4}), for each $u\in  H^{1}(\mathbb{R}^{2}, \mathbb{R})\backslash\{0\}$, there is a unique $t(u)>0$ such that
\begin{equation*}
I_{V_{0}}(t(u)u)=\max_{t\geq 0}I_{V_{0}}(tu)\quad\text{and}\quad t(u)u\in \mathcal{N}_{V_{0}}.
\end{equation*}
Then,  using the assumptions on $f$, arguing as in \cite[Lemma 4.1 and Theorem 4.2]{rW} we have that
\begin{equation*}
0<c_{V_{0}}=\inf_{u\in  H^{1}(\mathbb{R}^{2}, \mathbb{R})\backslash\{0\}} \max_{t\geq 0}I_{V_{0}}(tu).
\end{equation*}

Moreover, recalling that a positive ground state solution $\omega \in H^1(\mathbb{R}^2,\mathbb{R})$ of \eqref{3.2} satisfies $I_{V_0}(\omega)\leq I_{V_0}(v)$ for all positive nontrivial solutions $v \in H^1(\mathbb{R}^2,\mathbb{R})$ of \eqref{3.2},  by \cite[Corollary 1.5]{rASM}, we get
\begin{lemma}
Problem \eqref{3.2} has a positive  ground state solution $\omega\in H^1(\mathbb{R}^2,\mathbb{R})$ which is radially symmetric.
\end{lemma}
\begin{proof}
Let us recall that \cite[Corollary 1.5]{rASM} states that, if $h:\mathbb{R}\to\mathbb{R}$ is a continuous function satisfying:
\begin{enumerate}[label=(\roman{*}), ref=\roman{*}]
	\item \label{ASMi}$\lim_{t\to 0^+}h(t)/t=0$;
	\item \label{ASMii}there holds
	\[
	\lim_{t\rightarrow +\infty} \frac{h(t)}{e^{\alpha t^{2}}}
	=
	\begin{cases}
	0,& \hbox{for } \alpha> 4\pi,\\
	+\infty,& \hbox{for }  0<\alpha<4\pi;
	\end{cases}
	\]
	\item \label{ASMiii}there exist $\lambda>0$ and $p>2$ such that $h(t)\geq \lambda t^{p-1}$ for $t\geq 0$ and
	\begin{equation}
	\label{condlambda}
	\lambda>\left(\frac{p-2}{p}\right)^\frac{p-2}{2} S_p^{p/2},
	\end{equation}
	where $S_p$ is the best Sobolev constant for $S_p \| v\|_p^2 \leq \| v\|^2$;
\end{enumerate}
then
\begin{equation}
\label{ASM}
-\Delta v+v=h(v)
\end{equation}
has a nontrivial radial positive solution $\hat{v}\in H^1(\mathbb{R}^2,\mathbb{R})$, namely $\mathcal{I}(\hat{v})\leq \mathcal{I}(\tilde{v})$ for every nontrivial positive solution $\tilde{v}\in H^1(\mathbb{R}^2,\mathbb{R})$ of \eqref{ASM}, where
\[
\mathcal{I}(v):= \frac{1}{2} \|v\|^2 - \int_{\mathbb{R}^2} H(v) dx,
\quad
H(t):=\int_{0}^t h(s) ds.
\]
In particular, condition \eqref{condlambda} allows to prove that
\[
\bar{c}=\inf_{v\in  H^{1}(\mathbb{R}^{2}, \mathbb{R})\backslash\{0\}} \max_{t\geq 0} \mathcal{I}(tv)<\frac{1}{2}.
\]
Now, if we take
\begin{equation}
\label{gf}
h(t):= f(t^2)t/V_0
\end{equation}
in \eqref{ASM}, we have that (\ref{ASMi}) and (\ref{ASMii}) are easily satisfied, and, using assumption (\ref{f4}) we get that (\ref{ASMiii}) is satisfied for $\lambda={C_p}/{V_0}$.
Thus \eqref{ASM} with $h$ as in \eqref{gf} admits a positive radial nontrivial ground state solution $\hat{v}\in H^1(\mathbb{R}^2,\mathbb{R})$.\\
Observe now that, if $\hat{v}\in H^1(\mathbb{R}^2,\mathbb{R})$ is a solution of \eqref{ASM} where $h$ is given by \eqref{gf}, then $\hat{u}:=\hat{v}(\sqrt{V_0}\cdot)\in H^1(\mathbb{R}^2,\mathbb{R})$ is a solution of \eqref{3.2} and, since by \eqref{gf},
\[
H(t)=\frac{1}{V_0}\int_0^t f(s^2)s ds=\frac{1}{2V_0} F(t^2),
\]
we have $I_{V_0} (\tilde{u})=\mathcal{I}(\tilde{v})$.
Analogously, if $\tilde{u}\in H^1(\mathbb{R}^2,\mathbb{R})$ is an arbitrary solution of \eqref{3.2} and $\tilde{v}:=\tilde{u}(\cdot/\sqrt{V_0})$, then $\tilde{v}\in H^1(\mathbb{R}^2,\mathbb{R})$ is a solution of \eqref{ASM} and $I_{V_0} (\tilde{u})=\mathcal{I}(\tilde{v})$.\\
Hence, if $\hat{v}\in H^1(\mathbb{R}^2,\mathbb{R})$ is a positive radial nontrivial ground state of \eqref{ASM}, then, if $\omega=\hat{v}(\sqrt{V_0}\cdot)$, $\tilde{u}\in H^1(\mathbb{R}^2,\mathbb{R})$ is an arbitrary solution of \eqref{3.2} and $\tilde{v}:=\tilde{u}(\cdot/\sqrt{V_0})$, we have
\[
I_{V_0}(\omega)= \mathcal{I}(\hat{v})\leq \mathcal{I}(\tilde{v})=I_{V_0}(\tilde{u})
\]
and we conclude.
\end{proof}
Note that, by \cite[Proposition 2.1]{rZD},
every radially symmetric ground state
solution of \eqref{3.2}
decays exponentially at infinity with its gradient, and is $C^{2}(\mathbb{R}^{2}, \mathbb{R})\cap L^\infty(\mathbb{R}^{2}, \mathbb{R})$.

The elements of $\mathcal{N}_{V_{0}}$ satisfy the following property.
\begin{lemma}\label{leman}
There exists $K>0$ such that, for all $u\in \mathcal{N}_{V_{0}}$, $\| u\|_{V_0}\geq K$.
\end{lemma}
\begin{proof}
By \eqref{1.4}, for any $0<\zeta<V_{0}/2$  and $\alpha>4\pi$, we have that there exists $C>0$ such that, for every $u\in \mathcal{N}_{V_{0}}$,
\begin{equation}\label{Neh}
\int_{\mathbb{R}^{2}}(\vert \nabla u\vert^{2}+V_{0}\vert u\vert^{2})dx\leq \zeta\int_{\mathbb{R}^{2}}\vert u\vert^{2}+C\int_{\mathbb{R}^{2}}\vert u\vert^{q}(e^{\alpha  \vert u\vert^{2}}-1)dx.
\end{equation}
Moreover, by the H\"{o}lder inequality it follows
\begin{equation}
\label{3.1bisa}
\int_{\mathbb{R}^{2}}\vert u\vert^{q}(e^{\alpha \vert u\vert^{2}}-1)dx
\leq \Vert u\Vert_{2q}^{q}\Big(\int_{\mathbb{R}^{2}}(e^{\alpha \vert u\vert^{2}}-1)^{2}dx\Big)^{1/2}
\leq C\Vert u\Vert_{V_0}^{q}\Big(\int_{\mathbb{R}^{2}}(e^{2\alpha \vert u\vert^{2}}-1)dx\Big)^{1/2}
\end{equation}
where we have used the inequality
\begin{equation}
\label{ineqe}
(e^{t}-1)^{s}\leq e^{ts}-1, \text{ for }s>1 \text{ and } t\geq 0.
\end{equation}
Now assume by contradiction that there exist a sequence $(u_n)\subset \mathcal{N}_{V_{0}}$ such that $\| u_n\|_{V_0} \to 0$ as $n\to +\infty$. Then, for $n$ large enough and $\bar{\alpha}\in (0,4\pi)$, using Lemma \ref{le24}, we get
\begin{equation}
\label{Nehbis}
\int_{\mathbb{R}^{2}}(e^{2\alpha \vert u_n\vert^{2}}-1)dx
\leq
\int_{\mathbb{R}^{2}}(e^{2\alpha \| u_n\|_{V_0}^2 \left( \frac{u_n}{\| u_n\|_{V_0}}\right)^2} -1)dx
\leq
\int_{\mathbb{R}^{2}}(e^{\bar{\alpha} \left( \frac{u_n}{\| u_n\|_{V_0}}\right)^2} -1)dx
\leq C.
\end{equation}
Thus, combining \eqref{Neh}, \eqref{3.1bisa}, and \eqref{Nehbis}, we reach the contradiction.
\end{proof}

The following lemma is the upper bound estimate of the ground state energy which is important for our arguments.
\begin{lemma}\label{le32}
The minimax level $c_{V_{0}}$ verifies
\begin{equation*}
0<c_{V_{0}}<\frac{\theta-2}{2\theta}\min\{1, V_{0}\}.
\end{equation*}
\end{lemma}
\begin{proof}
Arguing as in \cite{rW}, we can find that there exists
$\omega^{*}\in H^{1}(\mathbb{R}^{2}, \mathbb{R})\setminus\{0\}$ such that $\tilde{I}_{0}(\omega^{*})=\beta_{p}$ and $\tilde{I}'_{0}(\omega^{*})=0$ (see (\ref{f4}) for the definitions of $\tilde{I}_{0}$ and $\beta_p$).
By the characterization of $c_{V_{0}}$ given before and by (\ref{f4}) we have
\begin{equation*}
0<c_{V_{0}}
\leq \max_{t\geq 0}I_{V_{0}}(t\omega^{*})
\leq \max_{t\geq 0}\Big\{\frac{t^{2}}{2}\| \omega^{*}\|_{V_0}^{2}
-\frac{C_{p}t^{p}}{p} \|\omega^{*}\|_p^{p}\Big\}
=C_{p}^{2/(2-p)}\beta_{p}
<\frac{\theta-2}{2\theta}\min\{1, V_{0}\}.
\end{equation*}
\end{proof}

Finally we prove the following useful result.
\begin{lemma}\label{lemFat}
Let $(\omega_{n})\subset \mathcal{N}_{V_{0}}$ be a sequence satisfying $I_{V_{0}}(\omega_{n})\rightarrow c_{V_{0}}$. Then $(\omega_{n})$ is bounded
in  $H^{1}(\mathbb{R}^{2}, \mathbb{R})$ and, up to a subsequence, $\omega_{n}\rightharpoonup \omega$ in $H^{1}(\mathbb{R}^{2}, \mathbb{R})$.
Moreover, if $\omega\neq 0$, then $\omega_{n}\rightarrow \omega\in \mathcal{N}_{V_{0}}$ in $H^{1}(\mathbb{R}^{2}, \mathbb{R})$ and $\omega$ is a ground state
for problem \eqref{3.2}.
If $\omega= 0$, then there exists $(\tilde{y}_{n})\subset \mathbb{R}^{2}$ with $\vert \tilde{y}_{n}\vert\rightarrow +\infty$ and $\tilde{\omega}\in \mathcal{N}_{V_{0}}$ such that, up to a subsequence, $\omega_{n}(\cdot+\tilde{y}_{n})\rightarrow \tilde{\omega}$ in $H^{1}(\mathbb{R}^{2}, \mathbb{R})$
and $\tilde{\omega}$ is a ground state for problem \eqref{3.2}.
\end{lemma}
\begin{proof}
By (\ref{f3}) and Lemma \ref{le32}, it follows that
 \begin{align*}
\frac{\theta-2}{2\theta} \limsup_n \|\omega_n\|_{V_0}^2
&\leq
\limsup_n\Big\{\Big(\frac{1}{2}-\frac{1}{\theta}\Big) \|\omega_n\|_{V_0}^2
+\int_{\mathbb{R}^{2}}\Big(\frac{1}{\theta}f( \omega_{n}^{2}) \omega_{n}^{2}-\frac{1}{2}F( \omega_{n}^{2})\Big)dx\Big\}\\
&=
\limsup_n \Big\{ I_{V_{0}}(\omega_{n})-\frac{1}{\theta}I'_{V_{0}}(\omega_{n})[\omega_{n}]\Big\}\\
&=
c_{V_{0}}<\frac{\theta-2}{2\theta}\min\{1, V_{0}\}.
\end{align*}
Thus,
\begin{equation}
\label{limsup<1}
\limsup_n \|\omega_n\|_{V_0}^2<1,
\end{equation}
and for some subsequence, still denoted by $(\omega_{n})$, we can assume that there exists $\omega\in H^{1}(\mathbb{R}^{2}, \mathbb{R})$ such that $\omega_{n}\rightharpoonup \omega$ in $H^{1}(\mathbb{R}^{2}, \mathbb{R})$, $\omega_{n}\rightarrow \omega$ in $L^{r}_{\rm loc}(\mathbb{R}^{2}, \mathbb{R})$, for any $r\geq 1$ and  $\omega_{n}\rightarrow \omega$ a.e. in $x\in \mathbb{R}^{2}$.\\
Now we divide our study into two cases.\\
{\bf Case 1:} $\omega\neq 0$.\\
Observe that, for every $\phi\in C_{c}^{\infty}(\mathbb{R}^{2},\mathbb{R})$,
$$f(\omega_{n}^{2})\omega_{n}\phi\to f(\omega^{2})\omega\phi
\text{ a.e. in } \mathbb{R}^2 \text{ as }n\to+\infty$$
and that, by \eqref{1.2}, we have that for any $\zeta>0$, $q>2$, and $\alpha>4\pi$, there exists $C>0$ such that, for every $\phi\in C_{c}^{\infty}(\mathbb{R}^{2},\mathbb{R})$,
\[
|f(\omega_{n}^{2})\omega_{n}\phi |
\leq \zeta \vert \omega_{n}\vert \vert \phi\vert
+C\vert \omega_{n}\vert^{q-1}(e^{\alpha \omega_{n}^{2}}-1)\vert \phi\vert
\]
with
\[
\zeta \vert \omega_{n}\vert \vert \phi\vert
+C\vert \omega_{n}\vert^{q-1}(e^{\alpha \omega_{n}^{2}}-1)\vert \phi\vert
\to
\zeta \vert \omega\vert \vert \phi\vert
+C\vert \omega \vert^{q-1}(e^{\alpha \omega^{2}}-1)\vert \phi\vert
\text{ a.e. in } \mathbb{R}^2 \text{ as }n\to+\infty.
\]
Moreover,
by the H\"older inequality, \eqref{ineqe}, Sobolev inequality, \eqref{limsup<1}, and Lemma \ref{le24}, for suitable $r>1$, $q>2$, $\alpha>4\pi$, and $p>1$, we have that, for all $n\in \mathbb{N}$,
	\begin{align*}
	\int_{\mathbb{R}^{2}}\Big[\vert \omega_{n}\vert^{q-1}(e^{\alpha \omega_{n}^{2}}-1)\Big]^r dx
	&\leq
	\|\omega_n\|_{pr(q-1)}^{r(q-1)}
	\Big( \int_{\mathbb{R}^{2}} (e^{\alpha r p' \omega_{n}^{2}}-1) dx\Big)^{1/p'}\\
	&\leq
	C
	\|\omega_n\|_{V_0}^{r(q-1)}
	\Big( \int_{\mathbb{R}^{2}} (e^{\alpha r p' \|\omega_n\|_{V_0}^2 (\vert\omega_{n}\vert/\|\omega_n\|_{V_0})^{2}}-1) dx\Big)^{1/p'}\\
	&\leq
	C.
	\end{align*}
		Thus
		$$\vert \omega_{n}\vert^{q-1}(e^{\alpha \vert\omega_{n}\vert^{2}}-1)\rightharpoonup \vert \omega\vert^{q-1}(e^{\alpha \vert\omega\vert^{2}}-1) \text{ in } L^{r}(\mathbb{R}^{2}, \mathbb{R})$$
		and, since $\vert \omega_{n}\vert \rightharpoonup \vert \omega\vert$ in $L^{2}(\mathbb{R}^{2}, \mathbb{R})$
		we have that
		\[
		\zeta\int_{\mathbb{R}^{2}}\vert \omega_{n}\vert \vert \phi\vert dx
		+C\int_{\mathbb{R}^{2}}\vert \omega_{n}\vert^{q-1}(e^{\alpha \omega_{n}^{2}}-1)\vert \phi\vert dx
		\to
		\zeta\int_{\mathbb{R}^{2}}\vert \omega \vert \vert \phi\vert dx
		+C\int_{\mathbb{R}^{2}}\vert \omega \vert^{q-1}(e^{\alpha \omega^{2}}-1)\vert \phi\vert dx
		\]
		as $n\to +\infty$.\\
		Hence, a variant of the Lebesgue Dominated Convergence Theorem
		implies that
		$$
		\int_{\mathbb{R}^{2}}f(\omega_{n}^{2})\omega_{n}\phi dx\rightarrow \int_{\mathbb{R}^{2}}f(\omega^{2})\omega\phi dx,
		$$
		and so $\omega$ is a nontrivial critical point for $I_{V_{0}}$.\\
Since,  by the Fatou's Lemma,
$$
\int_{\mathbb{R}^{2}}\Big(\frac{1}{\theta}f(\omega^{2})\omega^{2}-\frac{1}{2}F( \omega^{2})\Big)dx
\leq
\liminf_n\int_{\mathbb{R}^{2}}\Big(\frac{1}{\theta}f(\omega_{n}^{2}) \omega_{n}^{2}-\frac{1}{2}F(\omega_{n}^{2})\Big)dx,$$
we have
 \begin{align*}
c_{V_{0}}
&\leq
I_{V_{0}}(\omega)
=I_{V_{0}}(\omega)-\frac{1}{\theta}I'_{V_{0}}(\omega)[\omega]\\
&=
\Big(\frac{1}{2}-\frac{1}{\theta}\Big) \|\omega\|_{V_0}^2
+\int_{\mathbb{R}^{2}}\Big(\frac{1}{\theta}f( \omega^{2})\omega^{2}-\frac{1}{2}F( \omega^{2})\Big)dx\\
&\leq
\liminf_n\Big\{\Big(\frac{1}{2}-\frac{1}{\theta}\Big) \|\omega_n\|_{V_0}^2
+\int_{\mathbb{R}^{2}}\Big(\frac{1}{\theta}f( \omega_{n}^{2}) \omega_{n}^{2}-\frac{1}{2}F( \omega_{n}^{2})\Big)dx\Big\}\\
&=
\liminf_n \Big\{ I_{V_{0}}(\omega_{n})-\frac{1}{\theta}I'_{V_{0}}(\omega_{n})[\omega_{n}]\Big\}\\
&=c_{V_{0}}.
\end{align*}
Hence, using again the Fatou's Lemma, we have
 \begin{align*}
0
&\leq \liminf_n \left[\Big(\frac{1}{2}-\frac{1}{\theta}\Big)(\|\omega_n\|_{V_0}^2-\|\omega\|_{V_0}^2)\right]\\
&
\leq
\limsup_n \left[\Big(\frac{1}{2}-\frac{1}{\theta}\Big)(\|\omega_n\|_{V_0}^2-\|\omega\|_{V_0}^2)\right]
\\
&=
\limsup_n \Big[
I_{V_0}(\omega_n) - c_{V_0}
+\int_{\mathbb{R}^{2}}\Big(\frac{1}{\theta}f(\omega^{2})\omega^{2}-\frac{1}{2}F( \omega^{2})\Big)dx
-\int_{\mathbb{R}^{2}}\Big(\frac{1}{\theta}f(\omega_{n}^{2}) \omega_{n}^{2}-\frac{1}{2}F( \omega_{n}^{2})\Big)dx
\Big]
\\
&=
\int_{\mathbb{R}^{2}}\Big(\frac{1}{\theta}f(\omega^{2}) \omega^{2}-\frac{1}{2}F( \omega^{2})\Big)dx
-\liminf_n \Big[\int_{\mathbb{R}^{2}}\Big(\frac{1}{\theta}f(\omega_{n}^{2}) \omega_{n}^{2}-\frac{1}{2}F( \omega_{n}^{2})\Big)dx\Big]
\leq 0,
\end{align*}
and we conclude.\\
{\bf Case 2:} $\omega= 0$.\\
We claim that, in this case, there exist $R, \eta>0$, and $(\tilde{y}_{n})\subset \mathbb{R}^{2}$ such that
\begin{equation}\label{4.7}
\lim_n\int_{B_{R}(\tilde{y}_{n})}\omega_{n}^{2}dx\geq \eta.
\end{equation}
Indeed, if this does not hold, for any $R>0$, one has
\begin{equation*}
\lim_n \sup_{y\in \mathbb{R}^{2}}\int_{B_{R}(y)} \omega_{n}^{2}dx=0,
\end{equation*}
and by \cite[Chapter 6, Lemma 8.4]{rK}, for every $\tau>2$,
\begin{equation}
\label{limtau}
\lim_n \Vert \omega_{n}\Vert_{\tau}=0.
\end{equation}
By \eqref{1.4} and
the fact that $(\omega_{n})\subset \mathcal{N}_{V_{0}}$, for $0<\zeta<V_{0}/2$ and $\alpha>4\pi$, there exists $C>0$ such that
\begin{equation}\label{4.5C}
\int_{\mathbb{R}^{2}}(\vert \nabla \omega_{n}\vert^{2}+V_{0} \omega_{n}^{2})dx\leq \zeta\int_{\mathbb{R}^{2}} \omega_{n}^{2}+C\int_{\mathbb{R}^{2}}\vert \omega_{n} \vert^{q}(e^{\alpha   \omega_{n}^{2}}-1)dx.
\end{equation}
In virtue of  \eqref{limsup<1}, we may choose $r>1$ and $\alpha>4\pi$ such that 		$r\alpha\Vert \omega_{n}\Vert_{V_0}^{2}<4\pi$ for $n\in \mathbb{N}$ large enough. Thus, by the H\"{o}lder inequality, inequality \eqref{ineqe},  Lemma \ref{le24},  and \eqref{limtau}, it follows that
\begin{equation}
\label{3.5C}
\begin{split}
\int_{\mathbb{R}^{2}}\vert \omega_{n}\vert^{q}(e^{\alpha \omega_{n}^{2}}-1)dx
&\leq
\Vert \omega_{n}\Vert_{qr'}^{q}\Big(\int_{\mathbb{R}^{2}}(e^{r\alpha \omega_{n}^{2}}-1)dx\Big)^{1/r}\\
& \leq \Vert \omega_{n}\Vert_{qr'}^{q}\Big(\int_{\mathbb{R}^{2}}(e^{r\alpha\Vert \omega_{n}\Vert_{V_0}^{2}(\omega_{n} /\Vert\omega _{n}\Vert_{V_0})^{2}}-1)dx\Big)^{1/r}
=o_{n}(1),
\end{split}
\end{equation}
where $r'$ is the conjugate exponent of $r$.\\
Using \eqref{4.5C} and \eqref{3.5C},  we have that
$\omega_{n} \rightarrow 0 $ in $H^1(\mathbb{R}^2,\mathbb{R})$ as $n\rightarrow +\infty$, and,
consequently, $I_{V_{0}}(\omega_{n})\rightarrow 0$ as $n\rightarrow +\infty$, which is in contradiction with $I_{V_{0}}(\omega_{n})\rightarrow c_{V_{0}}>0$ as $n\rightarrow +\infty$.\\
By \eqref{4.7}, we have that $\vert \tilde{y}_{n}\vert\rightarrow+\infty$. Otherwise, there exists $\bar{R}>0$ such that
\begin{equation*}
\lim_n\int_{B_{\bar{R}}(0)}\omega_{n}^{2}dx\geq \eta
\end{equation*}
and so
$\omega\neq 0$, which is a contradiction.\\
Since $I_{V_{0}}$ and  the norm $\Vert \cdot \Vert_{V_{0}}$ in $H^1(\mathbb{R}^2,\mathbb{R})$ are invariant by translation, we have
\begin{equation*}
I_{V_{0}}(\omega_{n}(\cdot+\tilde{y}_{n}))\rightarrow c_{V_{0}}.
\end{equation*}
Moreover $\omega_{n}(\cdot+\tilde{y}_{n})\in \mathcal{N}_{V_{0}}$ and, by \eqref{limsup<1},
 \begin{align*}
\limsup_n \Vert  \omega_{n}(\cdot+\tilde{y}_{n}) \Vert_{V_{0}}^{2}<1.
\end{align*}
Thus, there exists $\tilde{\omega}\in H^{1}(\mathbb{R}^{2}, \mathbb{R})$ with, by \eqref{4.7}, $\tilde{\omega}\neq 0$, such that
\begin{align*}
 \omega_{n}(\cdot+\tilde{y}_{n})\rightharpoonup \tilde{\omega}\quad \text{in}\quad H^{1}(\mathbb{R}^{2}, \mathbb{R}).
\end{align*}
Repeating the same arguments used in Case 1, it is easy to obtain that $\omega_{n}(\cdot+\tilde{y}_{n})\rightarrow \tilde{\omega}\in \mathcal{N}_{V_{0}}$ in $H^{1}(\mathbb{R}^{2}, \mathbb{R})$ and $\tilde{\omega}$ is a ground state for problem \eqref{3.2}.
\end{proof}

\section{The modified problem}\label{sec3}

In this section we introduce a modified problem for \eqref{1.6} and we show some properties of its functional. As in \cite{rDF}, to study \eqref{1.1}, or equivalently, \eqref{1.6} by variational methods, we modify suitably the nonlinearity $f$ so that, for $\varepsilon>0$ small enough, the solutions of such modified problem are also solutions of the original one.
More precisely, we fix  $k>0$ such that
\begin{align*}
0
<c_{V_{0}}
<\Big(\frac{1}{2}-\frac{1}{\theta}-\frac{1}{2k}\Big)\min\{1, V_{0}\}
<\frac{\theta-2}{2\theta}\min\{1, V_{0}\}.
\end{align*}

By the assumptions on $f$ there exists a unique number $a>0$ verifying $k f(a)=V_{0}$, where $V_{0}$ is given in (\ref{V1}). Hence we consider the function
\begin{equation*}
	\hat{f}(t):=
	\begin{cases}
	f(t), & t\leq a,\\
	V_{0}/k, & t> a.
	\end{cases}
\end{equation*}%

As, for instance, in \cite{rAFF}, we take $0< t_{a}< a<T_{a}$ and $\vartheta\in C_{0}^{\infty}(\mathbb{R}, \mathbb{R})$ such that
\begin{enumerate}[label=($\vartheta_\arabic{*}$), ref=$\vartheta_\arabic{*}$]
	\item \label{theta1} $\vartheta(t)\leq \hat{f}(t)$ for all $t\in [t_{a}, T_{a}]$;
	\item \label{theta2} $\vartheta(t_{a})=\hat{f}(t_{a})$,  $\vartheta(T_{a})=\hat{f}(T_{a})$, $\vartheta'(t_{a})=\hat{f}'(t_{a})$, and $\vartheta'(T_{a})=\hat{f}'(T_{a})$;
	\item \label{theta3} the map $\vartheta$ is increasing in $[t_{a}, T_{a}]$.
\end{enumerate}

Using the above functions we can define $\tilde{f}\in C^{1}(\mathbb{R}, \mathbb{R})$ as follows
\begin{equation*}
	\tilde{f}(t):=
	\begin{cases}
	\hat{f}(t) & \text{if } t\not\in [t_{a}, T_{a}],\\
	\vartheta (t) & \text{if } t\in [t_{a}, T_{a}].
	\end{cases}
\end{equation*}

Now we introduce the penalized nonlinearity $g:\mathbb{R}^{2}\times \mathbb{R}\rightarrow \mathbb{R}$
\begin{equation}
	\label{1.7}
	g(x, t):=\chi_{\Lambda}(x)f(t)+(1-\chi_{\Lambda}(x))\tilde{f}(t),
\end{equation}
where $\chi_{\Lambda}$ is the characteristic function on $\Lambda$ and $\displaystyle G(x, t):=\int_{0}^{t}g(x, s)ds$.\\
In view of  (\ref{f1})--(\ref{f5}) and (\ref{theta1})--(\ref{theta3}), we have that $g$ is a Carath\'{e}odory function satisfying the following properties:
\begin{enumerate}[label=($g_\arabic{*}$), ref=$g_\arabic{*}$]
	\item \label{g1} $g(x, t)=0$ for each $t\leq 0$;
	\item \label{g2} $\underset{t\rightarrow 0^{+}}{\lim}g(x, t)=0$ uniformly in $x\in \mathbb{R}^{2}$;
	\item \label{g3} $g(x, t)\leq f(t)$ for all  $t\geq 0$ and uniformly in $x\in \mathbb{R}^{2}$;
	\item \label{g4} $0<\theta G(x, t)\leq 2 g(x, t)t$, for each $x\in \Lambda$, $t>0$;
	\item \label{g5} $0< g(x, t)\leq V_{0}/k$, for each $x\in \Lambda^{c}$, $t>0$;
	\item \label{g6} for each $x\in \Lambda$, the function  $t\mapsto g(x, t)$ is strictly increasing in $t\in (0, +\infty)$ and for each $x\in \Lambda^c$, the function  $t\mapsto g(x, t)$ strictly is increasing in $(0, t_{a})$.
\end{enumerate}

Then we consider the {\em modified} problem
\begin{equation}
	\label{1.8}
	\Big(\frac{1}{i}\nabla-A_{\varepsilon}(x)\Big)^{2}u
	+V_{\varepsilon}(x)u=g(\varepsilon x, \vert u\vert^{2})u
	\quad
	\hbox{in }\mathbb{R}^2.
\end{equation}
Note that, if  $u$
is a solution of problem \eqref{1.8} with
$$
	\vert u(x)\vert^2 \leq t_{a}\quad\text{for all } x\in \Lambda^{c}_{\varepsilon},
	\quad
	\Lambda_{\varepsilon}:=\{x\in \mathbb{R}^{2}: \varepsilon x\in \Lambda\},
$$
then $u$ is a solution of problem \eqref{1.6}.

The functional associated to problem \eqref{1.8} is
\begin{equation*}
J_{\varepsilon}(u)
:=\frac{1}{2}\int_{\mathbb{R}^{2}}(\vert \nabla_{A_\varepsilon}u\vert^{2}+V_\varepsilon(x)\vert u\vert^{2})dx
-\frac{1}{2}\int_{\mathbb{R}^{2}} G(\varepsilon x, \vert u\vert^{2})dx,
\end{equation*}
defined in $H_{\varepsilon}$.
It is standard to prove that $J_{\varepsilon}\in C^{1}(H_{\varepsilon}, \mathbb{R})$ and its critical points are nontrivial weak solutions of the modified
problem \eqref{1.8}.

Now we show that the functional $J_{\varepsilon}$ satisfies the Mountain Pass Geometry.
\begin{lemma}
\label{le31}
For any fixed $\varepsilon>0$, the functional $J_{\varepsilon}$ satisfies the following properties:
\begin{enumerate}[label=(\roman{*}), ref=\roman{*}]
	\item \label{MPGi} there exist $\beta, r>0$ such that $J_{\varepsilon}(u)\geq \beta$ if $\Vert u\Vert_{\varepsilon}=r$;
	\item \label{MPGii} there exists $e\in H_{\varepsilon}$ with $\Vert e\Vert_{\varepsilon}>r$ such that $J_{\varepsilon}(e)<0$.
\end{enumerate}
\end{lemma}
\begin{proof}
	Let us prove (\ref{MPGi}).\\
	By (\ref{g3}) and \eqref{1.5}, fixed $q>2$ and $\alpha>4\pi$, for any $\zeta>0$ and  there exists $C>0$ such that
\begin{equation}
\label{3.1}
G(\varepsilon x, \vert u\vert^{2})\leq\zeta \vert u\vert^{2}+C\vert u\vert^{q}(e^{\alpha \vert u\vert^{2}}-1) \quad \text{for all } x\in \mathbb{R}^{2}.
\end{equation}
By the H\"{o}lder and Sobolev inequalities and \eqref{ineqe} it follows
\begin{equation}
\label{3.1bis}
\int_{\mathbb{R}^{2}}\vert u\vert^{q}(e^{\alpha \vert u\vert^{2}}-1)dx
\leq \Vert u\Vert_{2q}^{q}\Big(\int_{\mathbb{R}^{2}}(e^{\alpha \vert u\vert^{2}}-1)^{2}dx\Big)^{1/2}
\leq C\Vert u\Vert_{\varepsilon}^{q}\Big(\int_{\mathbb{R}^{2}}(e^{2\alpha \vert u\vert^{2}}-1)dx\Big)^{1/2}.
\end{equation}
   Now, let us observe that, by the diamagnetic inequality \eqref{2.1}, if $u\in H_\varepsilon\setminus\{0\}$, it follows that
	$$\frac{\vert u\vert}{\Vert u\Vert_{\varepsilon}}\in H^{1}(\mathbb{R}^{2}, \mathbb{R}),
	\quad
	\Big\Vert \frac{\vert u\vert}{\Vert u\Vert_{\varepsilon}}\Big\Vert_2^2 \leq \frac{1}{V_0},
	\quad
	\Big\Vert \nabla \frac{ \vert u\vert}{\Vert u\Vert_{\varepsilon}}\Big\Vert_2^2\leq 1.$$
	Therefore, if we consider $\Vert u\Vert_{\varepsilon}=r>0$, for $\alpha r^2< \pi$,
	by Lemma \ref{le24}, there exists a constant $C>0$ such that
	\begin{equation}
	\label{3.1ter}
	\int_{\mathbb{R}^{2}}(e^{2\alpha \vert u\vert^{2}}-1)dx
	=\int_{\mathbb{R}^{2}}(e^{2\alpha r^{2}\big(\frac{\vert u\vert}{\Vert u\Vert_{\varepsilon}}\big)^{2}}-1)dx
	<\int_{\mathbb{R}^{2}}(e^{2\pi\big(\frac{\vert u\vert}{\Vert u\Vert_{\varepsilon}}\big)^{2}}-1)dx
	\leq C.
	\end{equation}
	Then, by \eqref{3.1}, \eqref{3.1bis}, and \eqref{3.1ter}, for any $\zeta>0$, there exits $C>0$ such that
	\begin{align*}
	J_{\varepsilon}(u)
	\geq\frac{1}{2}\Big(1-\frac{\zeta}{V_{0}}\Big)r^{2}-C r^{q}
	\end{align*}
	for any $u\in H_{\varepsilon}$ with $\Vert u\Vert_{\varepsilon}=r$ small enough and we can conclude easily since $q>2$.\\
To prove (\ref{MPGii}), let us fix $\varphi\in C_c^\infty(\mathbb{R}^2,\mathbb{C})\setminus\{0\}$ with $\operatorname{supp}(\varphi)\subset \Lambda_{\varepsilon}$. By (\ref{1.7}) and (\ref{f4})  we get
\begin{align*}
J_{\varepsilon}(t\varphi)\leq\frac{t^{2}}{2} \Vert \varphi \Vert^{2}_{\varepsilon}- \frac{C_{p}}{p}t^{p} \| \varphi\|_p^{p}
\end{align*}
and we can conclude passing to the limit as $t\rightarrow + \infty$, being $p>2$.
\end{proof}

Hence we can define the minimax level
\begin{equation*}
c_{\varepsilon}=\inf_{\gamma\in \Gamma_{\varepsilon}}\max_{t\in [0, 1]}J_{\varepsilon}(\gamma(t))
\end{equation*}
where
\begin{equation*}
\Gamma_{\varepsilon}=\{\gamma\in C([0, 1], H_{\varepsilon}): \gamma(0)=0 \text{ and } J_{\varepsilon}(\gamma(1))< 0\}.
\end{equation*}

The following results are important to prove the $(PS)_{c_{\varepsilon}}$ condition for the functional $J_{\varepsilon}$.
\begin{lemma}\label{le34}
Assume that $(u_{n})\subset H_{\varepsilon}$ is a $(PS)_{d}$ sequence for the functional $J_{\varepsilon}$. If
\begin{equation}
\label{estimceps}
0<d<\min\{1, V_{0}\}\Big(\frac{1}{2}-\frac{1}{\theta}-\frac{1}{2k}\Big),
\end{equation}
then $(u_{n})$ is bounded in $H_{\varepsilon}$ and
 \begin{align*}
\limsup_n\Vert \vert u_{n}\vert \Vert^{2}<1.
\end{align*}
\end{lemma}
\begin{proof}
By (\ref{g4}) and (\ref{g5}) we have
\begin{align*}
d+o_{n}(1)+o_{n}(1)\|u_n\|_\varepsilon
& \geq
J_{\varepsilon}(u_{n})
-\frac{1}{\theta}J_{\varepsilon}'(u_{n})[u_{n}]\\
&=
\Big(\frac{1}{2}-\frac{1}{\theta}\Big)\Vert u_{n}\Vert^{2}_{\varepsilon}
+ \int_{\mathbb{R}^{2}}\Big(\frac{1}{\theta}g(\varepsilon x, \vert u_{n}\vert^{2})\vert u_{n}\vert^{2}-\frac{1}{2}G(\varepsilon x, \vert u_{n}\vert^{2})\Big)dx\\
&\geq\Big(\frac{1}{2}-\frac{1}{\theta}\Big)\Vert u_{n}\Vert^{2}_{\varepsilon}
+\int_{\Lambda_{\varepsilon}^{c}}
\Big(\frac{1}{\theta}g(\varepsilon x, \vert u_{n}\vert^{2})\vert u_{n}\vert^{2}-\frac{1}{2}G(\varepsilon x, \vert u_{n}\vert^{2})\Big)dx\\
&\geq\Big(\frac{1}{2}-\frac{1}{\theta}\Big)\Vert u_{n}\Vert^{2}_{\varepsilon}-\frac{1}{2}\int_{\Lambda_{\varepsilon}^{c}}
G(\varepsilon x, \vert u_{n}\vert^{2})dx\\
&\geq\Big(\frac{1}{2}-\frac{1}{\theta}\Big)\Vert u_{n}\Vert^{2}_{\varepsilon}-\frac{1}{2k}\int_{\mathbb{R}^{2}}
V(\varepsilon x)\vert u_{n}\vert^{2}dx\\
&\geq\Big(\frac{1}{2}-\frac{1}{\theta}-\frac{1}{2k}\Big)\Vert u_{n}\Vert^{2}_{\varepsilon}.
\end{align*}
Thus $(u_n)$ is bounded in $H_\varepsilon$ and
\[
\Big(\frac{1}{2}-\frac{1}{\theta}-\frac{1}{2k}\Big)\Vert u_{n}\Vert^{2}_{\varepsilon}
\leq
d+o_{n}(1).
\]
Hence, by \eqref{estimceps} and the diamagnetic inequality \eqref{2.1} we have
\[
\min\{1, V_{0}\} \Big(\frac{1}{2}-\frac{1}{\theta}-\frac{1}{2k}\Big)
\limsup_n \| |u_n| \|^{2}
\leq
\Big(\frac{1}{2}-\frac{1}{\theta}-\frac{1}{2k}\Big)\limsup_n\Vert u_{n}\Vert^{2}_{\varepsilon}\\
\leq d<\min\{1, V_{0}\}\Big(\frac{1}{2}-\frac{1}{\theta}-\frac{1}{2k}\Big)
\]
and we can conclude.
\end{proof}

The next result is a version of the celebrated Lions Lemma (see e.g. \cite{rW}), which is useful in our arguments.
\begin{lemma}\label{le35}
Let $d>0$  and $(u_{n})\subset H_{\varepsilon}$ be a $(PS)_{d}$ sequence for $J_{\varepsilon}$ such that $u_{n}\rightharpoonup 0$ in $H_{\varepsilon}$ as $n\to +\infty$ and $\limsup_{n}\Vert \vert u_{n}\vert\Vert <1 $.  Then, one of the following alternatives occurs:
\begin{enumerate}[label=(\roman{*}), ref=\roman{*}]
	\item \label{le35i}$u_{n}\rightarrow 0$ in $H_{\varepsilon}$ as $n\to +\infty$;
	\item \label{le35ii} there are a sequence $(y_{n})\subset \mathbb{R}^{2}$ and constants $R$, $\beta>0$ such that
	\begin{equation*}
	\liminf_n \int_{B_{R}(y_{n})}\vert u_{n}\vert^{2}dx\geq \beta.
	\end{equation*}
\end{enumerate}
\end{lemma}
\begin{proof}
		Assume that (\ref{le35ii}) does not hold. Then, for every $R>0$, we have
		\begin{equation*}
		\lim_n \sup_{y\in \mathbb{R}^{2}}\int_{B_{R}(y)}\vert u_{n}\vert^{2}dx=0.
		\end{equation*}
		Being $(|u_{n}|)$ bounded in $H^1(\mathbb{R}^2)$, by \cite[Chapter 6, Lemma 8.4]{rK},   it follows that
		$\| u_{n}\|_\tau\rightarrow 0$ as $n\rightarrow+\infty$, for any $\tau>2$.\\
		Since, by Lemma \ref{le34}, $(u_{n})$ is a bounded $(PS)_{d}$ sequence for $J_{\varepsilon}$, then, using (\ref{g3}) and \eqref{1.4} we have that for any $\zeta>0$ there exists $C>0$ such that
		\begin{align*}
		0\leq \Vert u_{n}\Vert_{\varepsilon}^{2}
		&=
		\int_{\mathbb{R}^{2}}g(\varepsilon x, \vert u_{n}\vert^{2})\vert u_{n}\vert^{2}dx+o_{n}(1)\\
		&\leq \zeta \int_{\mathbb{R}^{2}}\vert u_{n}\vert^{2}dx+C\int_{\mathbb{R}^{2}}\vert u_{n}\vert^{q}(e^{\alpha\vert u_{n}\vert^{2}}-1)dx+o_{n}(1)\\
		&\leq \frac{\zeta}{V_{0}}\Vert u_{n}\Vert_{\varepsilon}^{2}+C\int_{\mathbb{R}^{2}}\vert u_{n}\vert^{q}(e^{\alpha\vert u_{n}\vert^{2}}-1)dx+o_{n}(1)
		\end{align*}
		for every $\alpha>4\pi$.\\
		Since $\limsup_n \Vert \vert u_{n}\vert\Vert<1 $, arguing as in the proof of Lemma \ref{le31}, we have that $\Vert u_{n}\Vert_{\varepsilon}\rightarrow 0$ in $H_{\varepsilon}$ and we conclude.
\end{proof}

The following lemma provides a range of levels in which the functional $J_{\varepsilon}$ verifies the Palais-Smale condition.
\begin{lemma}\label{le36}
The functional $J_{\varepsilon}$ satisfies the $(PS)_{d}$ condition at any level $0<d<\Big(\frac{1}{2}-\frac{1}{\theta}-\frac{1}{2k}\Big)\min\{1, V_{0}\}$.
\end{lemma}
\begin{proof} Let $(u_{n})\subset H_{\varepsilon}$ be a $(PS)_{d}$ for $J_{\varepsilon}$.
By Lemma \ref{le34}, $(u_{n})$ is bounded in $H_{\varepsilon}$ and $\limsup_{n}\Vert \vert u_{n}\vert\Vert<1$. Thus, up to a subsequence, $u_{n}\rightharpoonup u$
in $H_{\varepsilon}$ and $u_{n}\rightarrow u$ in $L^{q}_{\rm loc}(\mathbb{R}^{2}, \mathbb{R})$ for all $q\geq 1$ as $n\rightarrow+\infty$.
Moreover, by (\ref{g3}) and \eqref{1.2}, it follows that, fixed $q>2$, for any $\zeta>0$ and $\alpha>4\pi$, there exists a constant $C>0$, which depends on $q$, $\alpha$, $\zeta$, such that for every $\phi\in H_{\varepsilon}$,
\begin{align*}
\left|\operatorname{Re}\int_{\mathbb{R}^{2}}g(\varepsilon x, \vert u_{n}\vert^{2})u_{n}\overline{\phi}dx\right|
\leq \zeta\int_{\mathbb{R}^{2}}\vert u_{n}\vert \vert \overline{\phi}\vert dx+C
\int_{\mathbb{R}^{2}}\vert \overline{\phi}\vert\vert u_{n}\vert^{q-1}(e^{\alpha\vert u_{n}\vert^{2} }-1)dx.
\end{align*}
Arguing as in Lemma \ref{lemFat}, we have
\begin{align*}
\operatorname{Re}\int_{\mathbb{R}^{2}}g(\varepsilon x, \vert u_{n}\vert^{2})u_{n}\overline{\phi}dx\rightarrow \operatorname{Re}\int_{\mathbb{R}^{2}}g(\varepsilon x, \vert u\vert^{2})u\overline{\phi}dx.
\end{align*}
Thus, $u$ is a critical point of $J_{\varepsilon}$.\\
Let $R>0$ be such that $\Lambda_{\varepsilon}\subset B_{R/2}(0)$. We show that for any given $\zeta>0$, for $R$ large enough,
\begin{align}\label{3.8}
\limsup_n\int_{B_{R}^c(0)}(\vert \nabla_{A_\varepsilon}u_{n}\vert^{2}+V_{\varepsilon}(x)\vert u_{n}\vert^{2})dx\leq \zeta.
\end{align}
Let
$\phi_{R}\in C^{\infty}(\mathbb{R}^{2}, \mathbb{R})$ be a cut-off function
such that
\begin{equation*}
\phi_{R}=0\quad x\in B_{R/2}(0), \quad \phi_{R}=1\quad x\in B^{c}_{R}(0),\quad 0\leq \phi_{R}\leq 1,\quad\text{and}\, \quad\vert \nabla\phi_{R}\vert\leq C/R
\end{equation*}
where $C>0$ is a constant independent of $R$. Since the sequence $(\phi_{R}u_{n})$ is bounded in $H_{\varepsilon}$,  we have
\begin{equation*}
J_{\varepsilon}'( u_{n})(\phi_{R} u_{n})=o_{n}(1),
\end{equation*}
that is
\begin{align*}
\operatorname{Re}\int_{\mathbb{R}^{2}}\nabla_{A_\varepsilon}u_{n}\overline{\nabla_{A_\varepsilon}(\phi_{R}u_{n})}dx+\int_{\mathbb{R}^{2}}V_{\varepsilon}(x)\vert u_{n}\vert^{2}\phi_{R}dx=\int_{\mathbb{R}^{2}}g(\varepsilon x, \vert u_{n}\vert^{2})\vert u_{n}\vert^{2}\phi_{R}dx+o_{n}(1).
\end{align*}
Since $\overline{\nabla_{A_\varepsilon}(u_{n}\phi_{R})}=i\overline{u_{n}}\nabla\phi_{R}+\phi_{R}\overline{\nabla_{A_\varepsilon}u_{n}}$, using (\ref{g5}), we have
\begin{align*}
\int_{\mathbb{R}^{2}}(\vert \nabla_{A_\varepsilon}u_{n}\vert^{2}+V_{\varepsilon}(x)\vert u_{n}\vert^{2})\phi_{R}dx
&=\int_{\mathbb{R}^{2}}g( \varepsilon x, \vert u_{n}\vert^{2})\vert u_{n}\vert^{2}\phi_{R}dx
-\operatorname{Re}\int_{\mathbb{R}^{2}}i\overline{u_{n}}\nabla_{A_\varepsilon}u_{n}\nabla\phi_{R}dx
+o_{n}(1)\\
&\leq
\frac{1}{k}\int_{\mathbb{R}^{2}}V_{\varepsilon}(x)\vert u_{n}\vert^{2}\phi_{R}dx
-\operatorname{Re}\int_{\mathbb{R}^{2}}i\overline{u_{n}}\nabla_{A_\varepsilon}u_{n}\nabla\phi_{R}dx
+o_{n}(1).
\end{align*}
By the definition of $\phi_{R}$, the H\"{o}lder inequality and the boundedness of $(u_{n})$ in $H_{\varepsilon}$, we obtain
\begin{align*}
\Big(1-\frac{1}{k}\Big)\int_{\mathbb{R}^{2}}(\vert \nabla_{A_\varepsilon}u_{n}\vert^{2}+V_{\varepsilon}(x)\vert u_{n}\vert^{2})\phi_{R}dx\leq \frac{C}{R}\Vert {u}_{n}\Vert_{2}\Vert \nabla_{A_\varepsilon}u_{n}\Vert_{2}+o_{n}(1)\leq \frac{C_{1}}{R}+o_{n}(1)
\end{align*}
and so we can reach our claim.\\
Since $u_{n}\rightarrow  u$ in $L_{\rm loc}^{r}(\mathbb{R}^{2})$, for all $r\geq 1$, up to a subsequence, we have that
\begin{align*}
\vert u_{n}\vert\rightarrow \vert u\vert\text{ a.e. in }\mathbb{R}^2\text{ as }n\to+\infty.
\end{align*}
Then
\begin{align*}
g(\varepsilon x, \vert u_{n}\vert^{2})\vert u_{n}\vert^{2}\rightarrow g(\varepsilon x, \vert u\vert^{2})\vert u\vert^{2}  \text{ a.e. in } \mathbb{R}^2\text{ as }n\to+\infty.
\end{align*}
Moreover, $\vert u_{n}\vert\rightarrow  \vert u\vert$ in $L_{\rm loc}^{r}(\mathbb{R}^{2})$ for all $r\geq 1$.\\
Let
$$
P(x, t):=g(\varepsilon x, t^{2})t
\quad\text{and}\quad
Q(t):=e^{\alpha t^{2}}-1,
\quad t\in \mathbb{R},
$$
where $\alpha>4\pi$ with $\alpha\Vert \vert u_{n}\vert\Vert<4\pi$ for $n$ large.
Using (\ref{g3}) and (\ref{f2}), it is easy to see that
\begin{align*}
\underset{t\rightarrow +\infty}{\lim}\frac{P(x, t)}{Q(t)}=0\quad\text{uniformly for } x\in \mathbb{R}^2
\end{align*}
and, by Lemma \ref{le24},
\[
\sup_{n}\int_{\mathbb{R}^2}Q(\vert u_{n}\vert)dx\leq C.
\]
Then \cite[Theorem A.I]{rBL} implies
\begin{align*}
\lim_n \int_{B_{R}(0)}\Big|g(\varepsilon x, \vert u_{n}\vert^{2})\vert u_{n}\vert^{2} - g(\varepsilon x, \vert u\vert^{2})\vert u\vert^{2}\Big|dx=0.
\end{align*}
Moreover, by (\ref{g5}) and \eqref{3.8} we have
\[
\limsup_n \int_{B_{R}^c(0)}\Big\vert g(\varepsilon x, \vert u_{n}\vert^{2})\vert u_{n}\vert^{2}- g(\varepsilon x, \vert u\vert^{2})\vert u\vert^{2}\Big\vert dx
\leq
\limsup_n \frac{2}{k}\int_{ B_{R}^c(0)}(\vert \nabla_{A_\varepsilon}u_{n}\vert^{2}+V(\varepsilon x)\vert u_{n}\vert^{2})dx
< \frac{2\zeta}{k}
\]
for every $\zeta>0$.\\
Hence
	\begin{align*}
	\int_{\mathbb{R}^{2}}g(\varepsilon x, \vert u_{n}\vert^{2})\vert u_{n}\vert^{2}dx\rightarrow \int_{\mathbb{R}^{2}}g(\varepsilon x, \vert u\vert^{2})\vert u\vert^{2}dx \text{ as } n\rightarrow+\infty.
	\end{align*}
Finally, since $J_{\varepsilon}'(u)=0$,
we have
\[
o_{n}(1)
=J'_{\varepsilon}( u_{n}) [u_{n}]
=\Vert u_{n}\Vert_{\varepsilon}^{2}-\int_{\mathbb{R}^{2}}g(\varepsilon x, \vert u_{n}\vert^{2})\vert u_{n}\vert^{2}dx
=\Vert u_{n}\Vert_{\varepsilon}^{2}-\Vert u\Vert_{\varepsilon}^{2} + o_n(1).
\]
Thus, the sequence $(u_{n})$ strong converges to $u$ in $H_{\varepsilon}$.
\end{proof}

Since we would like to find multiple solutions of the functional $J_{\varepsilon}$, it is natural to consider it constrained to the Nehari manifold associated to our problem, that is
\begin{equation*}
 \mathcal{N}_{\varepsilon}:=\{u\in H_{\varepsilon}\backslash\{0\}: J'_{\varepsilon}(u)[u]=0\}.
\end{equation*}
In virtue of (\ref{g6}), it can be shown that for any $u\in H_\varepsilon\setminus\{0\}$, there exists a unique $t_\varepsilon>0$ such that
\[
\max_{t\geq 0}J_{\varepsilon}(tu)=J_{\varepsilon}(t_{\varepsilon}u)
\]
and $t_{\varepsilon}u\in \mathcal{N}_{\varepsilon}$. Thus, $c_{\varepsilon}$ can be characterized as follows
\begin{equation*}
c_{\varepsilon}=\inf_{u\in H_{\varepsilon}\backslash\{0\}}\sup_{t\geq 0}J_{\varepsilon}(tu)=\inf_{u\in \mathcal{N}_{\varepsilon}}J_{\varepsilon}(u).
\end{equation*}
Moreover, arguing as in Lemma \ref{leman}, we also have that there exists $\gamma>0$, which is independent of $\varepsilon>0$, such that
\begin{align}\label{3.14}
\Vert u\Vert_{\varepsilon}\geq \gamma>0,\quad\text{for each } u\in  \mathcal{N}_{\varepsilon}.
\end{align}

Now we show that  $\mathcal{N}_{\varepsilon}$ is a natural constraint, namely that the constrained critical points of the functional $J_{\varepsilon}$ on $\mathcal{N}_{\varepsilon}$ are the critical points of  $J_{\varepsilon}$ in $H_{\varepsilon}$.
First we prove the following property.

\begin{proposition}\label{prop31}
The functional $J_{\varepsilon}$ restricted to $\mathcal{N}_{\varepsilon}$ satisfies the $(PS)_{d}$ condition at any level $0<d<\Big(\frac{1}{2}-\frac{1}{\theta}-\frac{1}{2k}\Big)\min\{1, V_{0}\}$.
\end{proposition}
\begin{proof}
Let $(u_{n})\subset \mathcal{N}_{\varepsilon}$ be a $(PS)_{d}$ sequence of $J_{\varepsilon}$ restricted to $\mathcal{N}_{\varepsilon}$. Then,
$J_{\varepsilon}(u_{n})\rightarrow d$ as $n\rightarrow+\infty$ and there exists $(\lambda_{n})\subset \mathbb{R}$ such that
\begin{align}\label{3.15}
J'_{\varepsilon}(u_{n})=\lambda_{n}T'_{\varepsilon}(u_{n})+o_{n}(1),
\end{align}
where $T_{\varepsilon}:H_{\varepsilon}\rightarrow  \mathbb{R}$ is defined as
\begin{align*}
T_{\varepsilon}(u):=\Vert u\Vert_{\varepsilon}^{2}-\int_{\mathbb{R}^{2}}g(\varepsilon x, \vert u\vert^{2})\vert u\vert^{2} dx.
\end{align*}
Observe that, arguing as in Lemma \ref{le34}, we get that $(u_n)$ is bounded in $H_{\varepsilon}$ and $\limsup_n\Vert \vert u_{n}\vert \Vert^{2}<1$.\\
Note that, using the definition of $g$, the monotonicity of $\vartheta$, and (\ref{f4}),  we obtain
\begin{align*}
T'_{\varepsilon}(u_{n}) [u_{n}]
&=-2\int_{\mathbb{R}^{2}}g'(\varepsilon x, \vert u_{n}\vert^{2})\vert u_{n}\vert^{4} dx\leq -2\int_{\Lambda_{\varepsilon}\cup\{\vert u_n\vert^{2}<t_{a}\}}f'(\vert u_{n}\vert^{2})\vert u_{n}\vert^{4} dx\\
&\leq
-(p-2)C_{p}\int_{\Lambda_{\varepsilon}\cup\{\vert u_n\vert^{2}<t_{a}\}}\vert u_{n}\vert^{p} dx\leq -(p-2)C_{p}\int_{\Lambda_{\varepsilon}}\vert u_{n}\vert^{p} dx.
\end{align*}
Thus, up to a subsequence, we may assume that $ T'_{\varepsilon}(u_{n})[ u_{n}]\rightarrow\varsigma\leq 0$.\\
Let us prove that $\varsigma\neq 0$. Indeed, if $\varsigma=0$, then
\begin{align*}
o_{n}(1)=\vert T'_{\varepsilon}(u_{n}) [u_{n}] \vert\geq C \int_{\Lambda_{\varepsilon}}\vert u_{n}\vert^{p} dx.
\end{align*}
Thus we obtain that $u_{n}\rightarrow 0$ in $L^{p}(\Lambda_{\varepsilon}, \mathbb{C})$,  and by interpolation, we also have
$u_{n}\rightarrow 0$ in $L^{\tau}(\Lambda_{\varepsilon}, \mathbb{C})$, for all $\tau\geq 1$.
Moreover,
arguing as in Lemma \ref{le34}, we have that $\Vert \vert u_{n}\vert \Vert<1$ for $n$ large.  Hence, from
$J'_{\varepsilon}(u_{n})[u_{n}]=0$, (\ref{g3}), (\ref{g5}), \eqref{1.4}, the H\"{o}lder inequality and Lemma \ref{le24}, we conclude that
\begin{align*}
\Vert u_{n}\Vert_{\varepsilon}^{2}
&=\int_{\mathbb{R}^{2}}g(\varepsilon x,\vert u_{n}\vert^{2})\vert u_{n}\vert^{2} dx
\leq \int_{\Lambda_{\varepsilon}}f(\vert u_{n}\vert^{2})\vert u_{n}\vert^{2}dx+\frac{1}{k}\int_{\Lambda^{c}_{\varepsilon}}V(\varepsilon x)\vert u_{n}\vert^{2}dx\\
&\leq \zeta \int_{\Lambda_{\varepsilon}}\vert u_{n}\vert^{2}dx+C\int_{\Lambda_{\varepsilon}}\vert u_{n}\vert^{q}(e^{\alpha\vert u_{n}\vert^{2}}-1)dx+\frac{1}{k}\int_{\Lambda^{c}_{\varepsilon}}V(\varepsilon x)\vert u_{n}\vert^{2}dx\\
&=\frac{1}{k}\int_{\Lambda^{c}_{\varepsilon}}V(\varepsilon x)\vert u_{n}\vert^{2}dx+o_{n}(1),
\end{align*}
which implies that $u_{n}\rightarrow 0$ in $H_{\varepsilon}$. This is a contradiction with \eqref{3.14}. Therefore, $\varsigma<0$ and by \eqref{3.15} we deduce that
$\lambda_{n}=o_{n}(1)$.\\
On the other hand, since, by the definition of $g$ and (\ref{f5}), for every $\phi\in H_{\varepsilon}$ we have that
\begin{align*}
\int_{\mathbb{R}^{2}}g'(\varepsilon x, \vert u_{n}\vert^{2})\vert u_{n}\vert^{3} |\phi|  dx
&=
\int_{\Lambda_{\varepsilon}}f'(\vert u_{n}\vert^{2})\vert u_{n}\vert^{3} |\phi|  dx
+\int_{\Lambda^{c}_{\varepsilon}}\tilde{f}'(\vert u_{n}\vert^{2})\vert u_{n}\vert^{3} |\phi|  dx \\
&\leq
\int_{\Lambda_{\varepsilon}}(e^{4\pi |u_n|^2}-1) \vert u_{n}\vert^{3} |\phi|  dx
+\int_{\Lambda^{c}_{\varepsilon}\cap \{\vert u_{n}\vert^{2}\leq T_{a}\}}\tilde{f}'(\vert u_{n}\vert^{2})\vert u_{n}\vert^{3} |\phi|  dx\\
&\leq
\int_{\mathbb{R}^{2}}(e^{4\pi |u_n|^2}-1) \vert u_{n}\vert^{3} |\phi|  dx
+C\int_{\mathbb{R}^{2}}\vert u_{n}\vert^{3} |\phi|  dx,
\end{align*}
using (\ref{g3}), (\ref{1.2}), the fact that $\limsup_{n}\Vert \vert u_{n}\vert\Vert< 1$, the H\"{o}lder and Sobolev inequalities, for every $\phi\in H_{\varepsilon}$, we obtain
	\begin{align*}
	\vert T'_{\varepsilon}(u_{n}) [\phi] \vert
	&\leq
	2\Vert u_{n}\Vert_{\varepsilon}\Vert \phi\Vert_{\varepsilon}
	+2\int_{\mathbb{R}^{2}}g(\varepsilon x, \vert u_{n}\vert^{2})\vert u_{n}\vert\vert \phi\vert dx
	+2\int_{\mathbb{R}^{2}}g'(\varepsilon x, \vert u_{n}\vert^{2})\vert \vert u_{n}\vert^{3} \vert \phi\vert dx\\
	&\leq C
	\left[
	\Vert u_{n}\Vert_{\varepsilon}\Vert \phi\Vert_{\varepsilon}
	+\int_{\mathbb{R}^{2}} |u_n|^{q-1} (e^{\alpha |u_n|^2}-1) |\phi| dx
	+\int_{\mathbb{R}^{2}}(e^{4\pi |u_n|^2}-1) \vert u_{n}\vert^{3} \vert \phi\vert  dx\right.\\
	&\qquad\qquad\left.
	+\int_{\mathbb{R}^{2}}\vert u_{n}\vert^{3} |\phi|  dx
	\right] \\
	&\leq
	C(\Vert u_{n}\Vert_{\varepsilon}+\Vert u_{n}\Vert_{\varepsilon}^{q-1}+\Vert u_{n}\Vert_{\varepsilon}^{3})\Vert \phi\Vert_{\varepsilon}.
	\end{align*}
Then, the boundedness of $(u_{n})$ implies the boundedness of $T'_{\varepsilon}(u_{n})$ and so, by \eqref{3.15}, we can infer that $J'_{\varepsilon}(u_{n})=o_{n}(1)$,
that is $(u_{n})$ is a $(PS)_{d}$ sequence for $J_{\varepsilon}$. Hence, we apply Lemma \ref{le36} to conclude.
\end{proof}

As a consequence we get
\begin{corollary}\label{cor31}
The constrained critical points of the functional $J_{\varepsilon}$ on $\mathcal{N}_{\varepsilon}$ are critical points of $J_{\varepsilon}$ in $H_{\varepsilon}$.
\end{corollary}

\section{Multiple solutions for the modified problem }\label{sec4}

In this section, we prove a multiplicity result for the modified problem \eqref{1.8} using the Ljusternik-Schnirelmann category theory. In order to get it,
we first provide some useful preliminaries.

Let $\delta>0$ be such that $M_{\delta}\subset \Lambda$, $\omega\in H^{1}(\mathbb{R}^{2}, \mathbb{R})$ be a positive ground state solution of the limit problem \eqref{3.2}, and $\eta\in C^{\infty}(\mathbb{R}^{+}, [0, 1])$ be a nonincreasing cut-off function defined in $[0, +\infty)$ such that $\eta(t)=1$ if $0\leq t\leq \delta/{2}$
and $\eta(t)=0$ if $t\geq \delta$.\\
For any $y\in M$, let us introduce the function
\begin{equation*}
\Psi_{\varepsilon, y}(x):=\eta(\vert \varepsilon x-y\vert)\omega\Big(\frac{\varepsilon x-y}{\varepsilon}\Big)\exp\Big(i\tau_{y}\Big(\frac{\varepsilon x-y}{\varepsilon}\Big)\Big),
\end{equation*}
where $$\tau_{y}(x):=A_{1}(y)x_{1}+A_{2}(y)x_{2}.$$
Let $t_{\varepsilon}>0$ be the unique positive number such that
\begin{equation*}
\max_{t\geq 0}J_{\varepsilon}(t\Psi_{\varepsilon, y})=J_{\varepsilon}(t_{\varepsilon}\Psi_{\varepsilon, y}).
\end{equation*}
Note that $t_{\varepsilon}\Psi_{\varepsilon, y}\in \mathcal{N}_{\varepsilon}$.\\
Let us define $\Phi_{\varepsilon}:M\rightarrow \mathcal{N}_{\varepsilon}$ as
\begin{equation*}
\Phi_{\varepsilon}(y):=t_{\varepsilon}\Psi_{\varepsilon, y}.
\end{equation*}
By construction, $\Phi_{\varepsilon}(y)$ has compact support for any $y\in M$.\\
Moreover, the energy of the above functions has the following behavior as $\varepsilon\rightarrow 0^{+}$.
\begin{lemma}\label{le41}
The limit
\begin{equation*}
\lim_{\varepsilon\rightarrow 0^{+}}J_{\varepsilon}(\Phi_{\varepsilon}(y))=c_{V_{0}}
\end{equation*}
holds uniformly in $y\in M$.
\end{lemma}
\begin{proof}
Assume by contradiction that the statement is false. Then there exist $\delta_{0}>0$, $(y_{n})\subset M$ and $\varepsilon_{n}\rightarrow 0^+$
satisfying
\begin{equation*}
\Big\vert J_{\varepsilon_{n}}(\Phi_{\varepsilon_{n}}(y_{n}))-c_{V_{0}}\Big\vert\geq \delta_{0}.
\end{equation*}
For simplicity, we write $\Phi_{n}$, $\Psi_{n}$ and $t_{n}$  for $\Phi_{\varepsilon_{n}}(y_{n})$, $\Psi_{\varepsilon_{n}, y_{n}}$ and $t_{\varepsilon_{n}}$, respectively.\\
We can check that
\begin{equation}\label{4.1}
\Vert \Psi_{n}\Vert_{\varepsilon_{n}}^{2}\rightarrow \int_{\mathbb{R}^{2}}(\vert \nabla \omega\vert^{2}+V_{0} \omega^{2})dx
\text{ as } n\to+\infty.
\end{equation}
Indeed, by a change of variable of $z=(\varepsilon_{n} x-y_{n})/\varepsilon_{n}$, the Lebesgue Dominated Convergence Theorem, the continuity of $V$ and $y_{n}\in M\subset\Lambda$(which is bounded), we deduce that
\begin{equation*}
\int_{\mathbb{R}^{2}}V(\varepsilon_{n} x)\vert \Psi_{n}\vert^{2}dx= \int_{\mathbb{R}^{2}}V(\varepsilon_{n} z+y_{n})\vert \eta(\vert\varepsilon_{n} z \vert )\omega(z)\vert^{2}dx
\rightarrow V_{0} \int_{\mathbb{R}^{2}}\omega^{2}dx
\text{ as } n\to+\infty.
\end{equation*}
Moreover, by the same change of variable  $z=(\varepsilon_{n}x-y_n)/\varepsilon_{n}$, we also have
\begin{align*}
\int_{\mathbb{R}^{2}}\vert \nabla_{A_{\varepsilon_{n}}}\Psi_{n}\vert^{2}dx
&=
\varepsilon_{n}^2 \int_{\mathbb{R}^{2}}| \eta'(\vert \varepsilon_{n} z\vert)\omega(z)|^{2} dz
+\int_{\mathbb{R}^{2}} |  \eta(\vert \varepsilon_{n} z\vert)\nabla\omega(z)|^{2}dz\\
&\qquad
+\int_{\mathbb{R}^{2}}\Big|  \eta(\vert \varepsilon_{n} z\vert)\Big(A(y_n)-A(\varepsilon_{n}z+y_n)\Big)\omega(z)\Big|^{2}dz\\
&\qquad
+2\varepsilon_{n} \int_{\mathbb{R}^{2}}
\eta(\vert \varepsilon_{n} z\vert) \eta'(\vert \varepsilon_{n} z\vert)\omega(z)\nabla\omega(z)\cdot \frac{z}{|z|}dz.
\end{align*}
It is clear that
\begin{equation*}
\lim_n \int_{\mathbb{R}^{2}}|  \eta(\vert \varepsilon_{n} z\vert)\nabla\omega(z)|^{2}dz
=\int_{\mathbb{R}^{2}}\vert\nabla\omega(z)\vert^{2}dz.
\end{equation*}
Moreover, using the definition of $\eta$, the H\"{o}lder continuity with exponent $\alpha\in(0, 1]$ of $A$, the exponential decay of $\omega$, and the Lebesgue Dominated Convergence Theorem, we can infer
\begin{align*}
	\int_{\mathbb{R}^{2}}|\eta'(\vert \varepsilon_{n} z\vert)\omega(z)|^{2}dz
	=o_n(1),
	\end{align*}
	\begin{align*}
	\int_{\mathbb{R}^{2}}|\eta(\vert \varepsilon_{n} z\vert)\eta'(\vert \varepsilon_{n} z\vert)\omega(z)\nabla\omega(z)|dz
	=o_n(1),
	\end{align*}
	and
	\begin{align*}
	\int_{\mathbb{R}^{2}}\Big|  \eta(\vert \varepsilon_{n} z\vert)\Big(A(y_n)-A(\varepsilon_{n}z+y_n)\Big)\omega(z)\Big|^{2}dz\leq C\varepsilon_{n}^{2\alpha}\int_{\vert \varepsilon_{n} z\vert\leq \delta} \omega^{2}(z)\vert z\vert^{2\alpha}dz=o_{n}(1),
	\end{align*}
obtaining \eqref{4.1}.\\
On the other hand, since $J'_{\varepsilon_{n}}(t_{n}\Psi_{n})(t_{n}\Psi_{n})=0$, by the change of variables $z=(\varepsilon_{n} x-y_{n})/\varepsilon_{n} $, observe that, if $z\in B_{\delta/\varepsilon_{n}}(0)$, then $\varepsilon_{n} z+y_{n}\in B_{\delta}(y_{n})\subset M_{\delta}\subset\Lambda$,  
 we have
\begin{align*}
\Vert \Psi_{n}\Vert_{\varepsilon_{n}}^{2}
&=
\int_{\mathbb{R}^{2}} g(\varepsilon_{n}z+y_{n}, t_{n}^2\eta^2(\vert \varepsilon_{n} z\vert)\omega^2(z)) \eta^2(\vert \varepsilon_{n} z\vert)\omega^2(z) dz\\
&= \int_{\mathbb{R}^{2}}f(t_{n}^2\eta^2(\vert \varepsilon_{n} z\vert)\omega^2(z)) \eta^2(\vert \varepsilon_{n} z\vert)\omega^2(z) dz\\
&\geq
\int_{B_{\delta/(2\varepsilon_{n})}(0)}f( t_{n}^2 \omega^2(z))\omega^{2}(z)dz \\
&\geq \int_{B_{\delta/2}(0)}f( t_{n}^2\omega^2(z))\omega^{2}(z)dz\\
&\geq
f( t_{n}^2\gamma^{2})\int_{B_{\delta/2}(0)}\omega^{2}(z)dz
\end{align*}
for all $n$ large enough
and where $\gamma=\min\{\omega(z): \vert z\vert\leq \delta/2\}$.\\
If $t_{n}\rightarrow+\infty$, by (\ref{f4}) we deduce that $\Vert \Psi_{n}\Vert_{\varepsilon_{n}}^{2}\rightarrow+\infty$ which contradicts \eqref{4.1}.\\
Therefore, up to a subsequence, we may assume that $t_{n}\rightarrow t_{0}\geq 0$.\\
If $t_{n}\rightarrow  0$, using the fact that $f$ is increasing and the Lebesgue Dominated Convergence Theorem, we obtain that
$$\Vert \Psi_{n}\Vert_{\varepsilon_{n}}^{2}=\int_{\mathbb{R}^{2}}f(t_{n}^2 \eta^2(\vert \varepsilon_{n} z\vert)\omega^2(z)) \eta^2(\vert \varepsilon_{n} z\vert)\omega^2(z)dz\rightarrow 0, \text{ as } n\rightarrow+\infty,$$
which contradicts
\eqref{4.1}. Thus, we have $t_{0}>0$ and
\begin{equation*}
\int_{\mathbb{R}^{2}}(\vert \nabla \omega\vert^{2}+V_{0}\omega^2)dx=\int_{\mathbb{R}^{2}}f( t_{0}\omega^{2})\omega^{2}dx,
\end{equation*}
so that $t_{0}\omega\in \mathcal{N}_{V_{0}}$. Since $\omega\in \mathcal{N}_{V_{0}}$, we obtain that $t_{0}=1$ and so, using the Lebesgue Dominated Convergence Theorem, we get
\begin{equation*}
\lim_{n}\int_{\mathbb{R}^{2}}F(\vert t_{n}\Psi_{n}\vert^{2})dx=\int_{\mathbb{R}^{2}}F( \omega^{2})dx.
\end{equation*}
Hence
\begin{equation*}
\lim_{n}J_{\varepsilon_{n}}(\Phi_{\varepsilon_{n}}(y_{n}))=I_{V_{0}}(\omega)=c_{V_{0}}
\end{equation*}
which is a contradiction and conclude.
\end{proof}

Now we define the barycenter map.\\
Let $\rho>0$ be such that $M_{\delta}\subset B_{\rho}$ and consider
$\Upsilon: \mathbb{R}^{2}\rightarrow \mathbb{R}^{2}$ defined by setting
\begin{equation*}
\Upsilon(x):=\left\{
\begin{array}{l}
x,\quad \quad \quad \text{if}\,\,\vert x\vert<\rho,\\
\rho x/\vert x\vert ,\quad \text{if}\,\,\vert x\vert\geq\rho.
\end{array}%
\right.
\end{equation*}
The barycenter map $\beta_{\varepsilon}: \mathcal{N}_{\varepsilon}\rightarrow \mathbb{R}^{2}$ is defined by
\begin{equation*}
\beta_{\varepsilon}(u):=\frac{1}{\Vert u\Vert_2^{2}}\int_{\mathbb{R}^{2}}\Upsilon(\varepsilon x)\vert u(x)\vert^{2} dx.
\end{equation*}

We have
\begin{lemma}\label{le42}
The  limit
\begin{equation*}
\lim_{\varepsilon\rightarrow 0^{+}}\beta_{\varepsilon}(\Phi_{\varepsilon}(y))=y
\end{equation*}
holds uniformly in $ y\in M$.
\end{lemma}
\begin{proof}
Assume by contradiction that there exists $\kappa>0$, $(y_{n})\subset M$ and $\varepsilon_{n}\rightarrow 0$ such that
\begin{equation}\label{4.2}
\vert \beta_{\varepsilon_{n}}(\Phi_{\varepsilon_{n}}(y_{n}))-y_{n}\vert \geq\kappa.
\end{equation}
Using the change of variable $z=(\varepsilon_{n} x-y_{n})/\varepsilon_{n}$, we can see that
\begin{equation*}
\beta_{\varepsilon_{n}}(\Phi_{\varepsilon_{n}}(y_{n}))=y_{n}+\frac{\displaystyle\int_{\mathbb{R}^{2}}(\Upsilon(\varepsilon_{n} z+y_{n})-y_{n}) \eta^2(\vert \varepsilon_{n} z\vert)  \omega^2(z) dz}{\displaystyle\int_{\mathbb{R}^{2}}\eta^2(\vert \varepsilon_{n} z\vert) \omega^2(z) dz}.
\end{equation*}
Taking into account $(y_{n})\subset M\subset M_{\delta}\subset B_{\rho}$ and the Lebesgue Dominated Convergence Theorem, we can obtain that
\begin{equation*}
\vert \beta_{\varepsilon_{n}}(\Phi_{\varepsilon_{n}}(y_{n}))-y_{n}\vert=o_{n}(1),
\end{equation*}
which contradicts \eqref{4.2}.
\end{proof}

Now, we prove the following useful compactness result.

\begin{proposition}\label{prop41}
Let $\varepsilon_{n}\rightarrow 0^+$ and $(u_{n})\subset \mathcal{N}_{\varepsilon_{n}}$ be such that $J_{\varepsilon_{n}}(u_{n})\rightarrow c_{V_{0}}$. Then there exists
$(\tilde{y}_{n})\subset \mathbb{R}^{2}$ such that the sequence $(\vert v_n\vert)\subset H^{1}(\mathbb{R}^{2}, \mathbb{R})$, where $v_{n}(x):= u_{n}(x+\tilde{y}_{n})$, has a convergent subsequence in $H^{1}(\mathbb{R}^{2}, \mathbb{R})$. Moreover, up to a subsequence, $y_{n}:=\varepsilon_{n}\tilde{y}_{n}\rightarrow y\in M$ as $n\to+\infty$.
\end{proposition}
\begin{proof}
Since $J'_{\varepsilon_{n}}(u_{n}) [u_{n}]=0$ and $J_{\varepsilon_{n}}(u_{n})\rightarrow c_{V_{0}}$, arguing as in the proof of Lemma \ref{le34},
using Lemma \ref{le32}, we can prove that there exists
$C>0$ such that $\Vert u_{n}\Vert_{\varepsilon_{n}}\leq C$ for all $n\in \mathbb{N}$ and  ${\limsup_n}\Vert \vert u_{n}\vert \Vert<1$.\\
Arguing as in the proof of Lemma \ref{le35} and recalling that $c_{V_{0}}>0$, we have that there exist a sequence $(\tilde{y}_{n})\subset \mathbb{R}^{2}$ and constants $R$, $\beta>0$ such that
\begin{equation}\label{4.3}
{\liminf_n}\int_{B_{R}(\tilde{y}_{n})}\vert u_{n}\vert^{2}dx\geq \beta.
\end{equation}
%
Now, let us consider the sequence $(\vert v_n\vert)\subset H^{1}(\mathbb{R}^{2}, \mathbb{R})$, where $v_{n}(x):= u_{n}(x+\tilde{y}_{n})$.\\
By the diamagnetic inequality \eqref{2.1}, we get that $(\vert v_{n}\vert)$ is bounded in $H^{1}(\mathbb{R}^{2}, \mathbb{R})$, and using \eqref{4.3}, we may
assume that $\vert v_{n}\vert\rightharpoonup v$ in $H^{1}(\mathbb{R}^{2}, \mathbb{R})$ for some $v\neq 0$.\\
Let now $t_{n}>0$ be such that $\tilde{v}_{n}:=t_{n}\vert v_{n}\vert\in \mathcal{N}_{V_{0}}$, and set $y_{n}:=\varepsilon_{n}\tilde{y}_{n}$.\\
By the diamagnetic inequality \eqref{2.1}, we have
\begin{equation*}
c_{V_{0}}\leq I_{V_{0}}(\tilde{v}_{n})\leq \max_{t\geq 0}J_{\varepsilon_{n}}(tu_{n})=J_{\varepsilon_{n}}(u_{n})=c_{V_{0}}+o_{n}(1),
\end{equation*}
which yields $I_{V_{0}}(\tilde{v}_{n})\rightarrow c_{V_{0}}$ as $n\to+\infty$.\\
Since the sequences  $(\vert v_{n}\vert)$ and $(\tilde{v}_{n})$ are bounded in $H^{1}(\mathbb{R}^{2}, \mathbb{R})$ and $\vert v_{n}\vert\not\rightarrow 0$ in $H^{1}(\mathbb{R}^{2}, \mathbb{R})$, then $(t_{n})$ is also bounded and so, up to a subsequence, we may assume that
$t_{n}\rightarrow t_{0}\geq 0$.\\
We claim that $t_{0}>0$.
Indeed, if $t_0=0$, then, since $(|v_{n}|)$ is bounded, we have $\tilde{v}_{n}\rightarrow 0$ in $H^{1}(\mathbb{R}^{2}, \mathbb{R})$, that is $I_{V_{0}}(\tilde{v}_{n})\rightarrow 0$, which contradicts $c_{V_{0}}>0$.\\
Thus, up to a subsequence, we may assume that $\tilde{v}_{n}\rightharpoonup \tilde{v}:=t_{0} v\neq 0$ in $H^{1}(\mathbb{R}^{2}, \mathbb{R})$, and, by Lemma \ref{lemFat},
we can deduce that $\tilde{v}_{n}\rightarrow \tilde{v}$ in $H^{1}(\mathbb{R}^{2}, \mathbb{R})$, which gives $\vert v_{n}\vert\rightarrow v$ in $H^{1}(\mathbb{R}^{2}, \mathbb{R})$.\\
Now we show the final part, namely that $(y_{n})$ has a subsequence such that $y_{n}\rightarrow y\in M$.\\
Assume by contradiction that $(y_{n})$ is not bounded and so, up to a subsequence, $\vert y_{n}\vert\rightarrow +\infty$ as $n\rightarrow+\infty$. Choose $R>0$ such that $\Lambda\subset B_{R}(0)$. Then for $n$ large enough, we have $\vert y_{n}\vert>2R$,  and, for any $x\in B_{R/\varepsilon_{n}}(0)$,
\begin{equation*}
\vert \varepsilon_{n} x+y_{n}\vert\geq \vert y_{n}\vert-\varepsilon_{n}\vert  x\vert>R.
\end{equation*}
Since $u_{n}\in \mathcal{N}_{\varepsilon_{n}}$, using (\ref{V1}) and the diamagnetic inequality \eqref{2.1}, we get that
\begin{equation}\label{4.4}
\begin{split}
\int_{\mathbb{R}^{2}}(\vert\nabla   \vert v_{n}\vert\vert^{2}+ V_{0}\vert v_{n}\vert^{2})dx&\leq \int_{\mathbb{R}^{2}}g(\varepsilon_{n} x+y_{n}, \vert v_{n}\vert^{2})\vert v_{n}\vert^{2}dx\\
&\leq \int_{B_{R/\varepsilon_{n}}(0)}\tilde{f}(\vert v_{n}\vert^{2})\vert v_{n}\vert^{2}dx+\int_{B^{c}_{R/\varepsilon_{n}}(0)}f(\vert v_{n}\vert^{2})\vert v_{n}\vert^{2}dx.
\end{split}
\end{equation}
Since $\vert v_{n}\vert\rightarrow v$ in $H^{1}(\mathbb{R}^{2}, \mathbb{R})$ and $\tilde{f}(t)\leq V_{0}/{k}$, we can see that \eqref{4.4} yields
\begin{align*}
\min\Big\{1, V_{0}\Big(1-\frac{1}{k}\Big)\Big\}\int_{\mathbb{R}^{2}}(\vert\nabla  \vert v_{n}\vert\vert^{2}+ \vert v_{n}\vert^{2})dx=o_{n}(1),
\end{align*}
that is $\vert v_{n}\vert\rightarrow 0$ in $H^{1}(\mathbb{R}^{2}, \mathbb{R})$, which contradicts to $v\not\equiv 0$.\\
Therefore, we may assume that
$y_{n}\rightarrow y_{0}\in \mathbb{R}^{2}$.\\
Assume  by contradiction that $ y_{0}\not\in \overline{\Lambda}$.  Then there exists $r>0$ such that for every $n$ large enough we have that $|y_n -y_0|<r$ and $B_{2r}(y_0)\subset\overline{\Lambda}^c$. Then, if $x\in B_{r/\varepsilon_n}(0)$, we have that $|\varepsilon_n x + y_n - y_0| <2r$ so that $\varepsilon_n x + y_n\in \overline{\Lambda}^c$ and so, arguing as before, we reach a contradiction.\\	
Thus, $y_{0}\in \overline{\Lambda}$.\\
To prove that $V(y_{0})=V_{0}$,  we suppose by contradiction that $V(y_{0})>V_{0}$. Using the Fatou's lemma, the change of variable $z=x+\tilde{y}_n$ and $\max_{t\geq 0}J_{\varepsilon_{n}}(t u_{n})=J_{\varepsilon_{n}}(u_{n})$,
we obtain
\begin{align*}
	c_{V_{0}}=I_{V_{0}}(\tilde{v})&<\frac{1}{2}\int_{\mathbb{R}^{2}}(\vert\nabla  \tilde{v}\vert^{2}+ V(y_{0})\vert \tilde{v}\vert^{2})dx-\frac{1}{2}\int_{\mathbb{R}^{2}}F(\vert \tilde{v}\vert^{2})dx\\
	&\leq {\liminf_n}\Big(\frac{1}{2}\int_{\mathbb{R}^{2}}(\vert\nabla \tilde{v}_{n}\vert^{2}+ V(\varepsilon_{n}x+y_{n})\vert \tilde{v}_{n}\vert^{2})dx-\frac{1}{2}\int_{\mathbb{R}^{2}}F(\vert \tilde{v}_{n}\vert^{2})dx\Big)\\
	&={\liminf_n}\Big(\frac{t^{2}_{n}}{2}\int_{\mathbb{R}^{2}}(\vert\nabla  \vert u_{n}\vert\vert^{2}+ V(\varepsilon_{n}z)\vert u_{n}\vert^{2})dz-\frac{1}{2}\int_{\mathbb{R}^{2}}F(\vert t_{n}u_{n}\vert^{2})dz\Big)\\
	&\leq {\liminf_n}J_{\varepsilon_{n}}(t_{n}u_{n})\leq {\liminf_n}J_{\varepsilon_{n}}(u_{n})= c_{V_{0}}
\end{align*}
which is impossible and we conclude.
\end{proof}

Let now
\begin{equation*}
\tilde{\mathcal{N}}_{\varepsilon}:=\{u\in \mathcal{N}_{\varepsilon}: J_{\varepsilon}(u)\leq c_{V_{0}}+h(\varepsilon)\},
\end{equation*}
where $h: \mathbb{R}^{+}\rightarrow \mathbb{R}^{+}$, $h(\varepsilon)\rightarrow 0$ as $\varepsilon\rightarrow 0^+$.\\
Fixed $y\in M$, since, by Lemma \ref{le41}, $\vert J_{\varepsilon}(\Phi_{\varepsilon}(y))-c_{V_{0}}\vert\rightarrow 0$ as $\varepsilon\rightarrow 0^+$, we get that $\tilde{\mathcal{N}}_{\varepsilon}\neq\emptyset$ for any $\varepsilon>0$ small enough.

We have the following relation between $\tilde{\mathcal{N}}_{\varepsilon}$ and the barycenter map.
\begin{lemma}\label{le43}
We have
\begin{equation*}
\lim_{\varepsilon\rightarrow 0^{+}}\sup_{u\in \tilde{\mathcal{N}}_{\varepsilon}}\operatorname{dist}(\beta_{\varepsilon}(u), M_{\delta})=0.
\end{equation*}
\end{lemma}
\begin{proof}
Let $\varepsilon_{n}\rightarrow 0^+$ as $n\rightarrow+\infty$. For any $n\in \mathbb{N}$, there exists $u_{n}\in\tilde{\mathcal{N}}_{\varepsilon_n}$ such that
\begin{equation*}
\sup_{u\in \tilde{\mathcal{N}}_{\varepsilon_n}}\inf_{y\in  M_{\delta}}|\beta_{\varepsilon_n}(u)- y|=\inf_{y\in  M_{\delta}}|\beta_{\varepsilon_n}(u_{n})- y|+o_{n}(1).
\end{equation*}
Therefore, it is enough to prove that there exists $(y_{n})\subset M_{\delta}$ such that
 \begin{equation*}
\lim_{n}\vert\beta_{\varepsilon_n}(u_{n})-y_{n}\vert=0.
\end{equation*}
By the diamagnetic inequality \eqref{2.1}, we can see that $I_{V_{0}}(t\vert u_{n}\vert)\leq J_{\varepsilon_{n}}(t u_{n})$ for any $t\geq 0$. Therefore, recalling
that $(u_{n})\subset \tilde{\mathcal{N}}_{\varepsilon_n}\subset \mathcal{N}_{\varepsilon_n}$, we can deduce that
 \begin{equation}\label{4.6}
c_{V_{0}}\leq\max_{t\geq 0}I_{V_{0}}(t\vert u_{n}\vert)\leq \max_{t\geq 0}J_{\varepsilon_{n}}(t u_{n})=J_{\varepsilon_{n}}(u_{n})\leq c_{V_{0}}+h(\varepsilon_{n})
\end{equation}
which implies that $J_{\varepsilon_{n}}(u_{n})\rightarrow c_{V_{0}}$ as $n\rightarrow+\infty$.\\
Then, Proposition \ref{prop41} implies that there exists $(\tilde{y}_{n})\subset \mathbb{R}^{2}$ such that $y_{n}=\varepsilon_{n}\tilde{y}_{n}\in M_{\delta}$ for $n$ large enough.\\
Thus, making the change of variable $z=x-\tilde{y}_n$, we get
\begin{equation*}
\beta_{\varepsilon_{n}}(u_{n})=y_{n}+\frac{\int_{\mathbb{R}^{2}}(\Upsilon(\varepsilon_{n} z+y_{n})-y_{n})\vert u_{n}(z+\tilde{y}_{n})\vert^{2} dz}{\int_{\mathbb{R}^{2}}\vert  u_{n}(z+\tilde{y}_{n})\vert^{2} dz}.
\end{equation*}
Since, up to a subsequence, $\vert u_{n}\vert(\cdot+\tilde{y}_{n})$ converges strongly in $H^{1}(\mathbb{R}^{2}, \mathbb{R})$ and $\varepsilon_{n}z+y_{n}\rightarrow y\in M$ for any $z\in \mathbb{R}^{2}$, we conclude.
\end{proof}

Finally, we present a relation between the topology of $M$ and the number of solutions of the modified problem \eqref{1.8}.
\begin{theorem}\label{th41}
For any $\delta>0$ such that $M_{\delta}\subset \Lambda$, there exists $\tilde{\varepsilon}_{\delta}>0$ such that, for any $\varepsilon\in (0, \tilde{\varepsilon}_{\delta})$, problem
\eqref{1.8} has at least $\text{cat}_{M_{\delta}}(M)$ nontrivial solutions.
\end{theorem}
\begin{proof}
Given $\delta>0$, by Lemma \ref{le41}, Lemma \ref{le42}, and Lemma \ref{le43}, and arguing as in \cite[Section~6]{rCL}, we can find $\tilde{\varepsilon}_{\delta}>0$ such that for any $\varepsilon\in (0, \tilde{\varepsilon}_{\delta})$, the following diagram
\begin{equation*}
M \xrightarrow{\Phi_{\varepsilon}}\tilde{\mathcal{N}}_{\varepsilon}\xrightarrow{\beta_{\varepsilon}} M_{\delta}
\end{equation*}
is well defined and $\beta_{\varepsilon}\circ \Phi_{\varepsilon}$ is homotopically equivalent to the embedding $\iota: M\rightarrow M_{\delta}$. Thus, \cite[Lemma~4.3]{rBC} (see also \cite[Lemma 2.2]{rCLJDE}) implies that
\begin{equation*}
\text{cat}_{\tilde{\mathcal{N}}_{\varepsilon}}(\tilde{\mathcal{N}}_{\varepsilon})\geq \text{cat}_{M_{\delta}}(M).
\end{equation*}
By Proposition \ref{prop31}, we have also that $J_{\varepsilon}$ satisfies the Palais-Smale condition on $\tilde{\mathcal{N}}_{\varepsilon}$ (taking $\tilde{\varepsilon}_{\delta}$ smaller if necessary).
Hence, by the Ljusternik-Schnirelmann theory for $C^{1}$ functionals (see \cite[Theorem~5.20]{rW}), we get at least $\text{cat}_{M_{\delta}}(M)$ critical points of $J_{\varepsilon}$ restricted to $\mathcal{N}_{\varepsilon}$ which are, by Corollary \ref{cor31}, critical points for $J_{\varepsilon}$  in $\tilde{\mathcal{N}}_{\varepsilon}$. 
\end{proof}

\section{Proof of Theorem \ref{mt} }\label{Sec5}

In this section we prove our main result. The idea is to show that the solutions $u_\varepsilon$ obtained in Theorem \ref{th41} satisfy
\[
\vert u_{\varepsilon}(x)\vert^2\leq t_{a}
\text{ for } x\in \Lambda^{c}_{\varepsilon}
\]
for $\varepsilon$ small. The key ingredient is the following result.
\begin{lemma}\label{le51}
Let $\varepsilon_{n}\rightarrow 0^{+}$ and $u_{n}\in \tilde{\mathcal{N}}_{\varepsilon_{n}}$ be a solution of problem \eqref{1.8} for $\varepsilon=\varepsilon_n$. Then $J_{\varepsilon_{n}}(u_{n})\rightarrow c_{V_{0}}$.
Moreover, there exists $(\tilde{y}_{n})\subset\mathbb{R}^2$ such that, if $v_{n}(x):= u_{n}(x+\tilde{y}_{n})$,  we have that $(\vert v_{n}\vert)$ is bounded in $L^{\infty}(\mathbb{R}^{2}, \mathbb{R})$ and
\begin{equation*}
\lim_{\vert x\vert\rightarrow+\infty} \vert v_{n}(x)\vert=0\quad\text{uniformly in } n\in \mathbb{N}.
\end{equation*}
\end{lemma}

\begin{proof}
Since $J_{\varepsilon_{n}}(u_{n})\leq c_{V_{0}}+h(\varepsilon_{n})$ with $\lim_{n}h(\varepsilon_{n})=0$, we can argue as in the proof of Lemma \ref{le43} (see \eqref{4.6}) to conclude that  $J_{\varepsilon_{n}}(u_{n}) \rightarrow c_{V_{0}}$.\\
Thus, by Proposition \ref{prop41}, we obtain the existence of a sequence $(\tilde{y}_{n})\subset \mathbb{R}^{2}$ such that  $(\vert v_n\vert)\subset H^{1}(\mathbb{R}^{2}, \mathbb{R})$, where $v_{n}(x):=u_{n}(x+\tilde{y}_{n})$, has a convergent subsequence in $H^{1}(\mathbb{R}^{2}, \mathbb{R})$. Moreover, up to a subsequence, $y_{n}:=\varepsilon_{n}\tilde{y}_{n}\rightarrow y\in M$ as $n\to+\infty$.\\
For any $R>0$ and $0<r\leq R/2$, let $\eta\in C^{\infty}(\mathbb{R}^{2})$, $0\leq\eta\leq 1$ with $\eta(x)=1$ if $\vert x\vert\geq R$
and $\eta(x)=0$ if $\vert x\vert\leq R-r$ and
$\vert \nabla \eta\vert\leq 2/r$.\\
For each $n\in \mathbb{N}$ and  $L>0$, we consider the functions
\begin{equation*}
v_{L, n}(x):=
\begin{cases}
\vert v_{n}(x)\vert & \text{if }\vert v_{n}(x)\vert \leq L,\\
L & \text{if }\vert v_{n}(x)\vert > L,
\end{cases}
\qquad
z_{L, n}:=\eta^{2}v_{L, n}^{2(\beta-1)}v_{n},
\quad \text{and}\quad
w_{L, n}:=\eta v_{L, n}^{\beta-1}\vert v_{n}\vert,
\end{equation*}
where $\beta>1$ will be determined later.\\

	Since,
%
	by the diamagnetic inequality \eqref{2.1} we have that
	\begin{align*}
	\operatorname{Re}(\nabla_{A_{\varepsilon_{n}}(\cdot+ \tilde{y}_n)}v_{n} \cdot \overline{\nabla_{A_{\varepsilon_{n}}(\cdot+ \tilde{y}_n)}z_{L, n}})
	&=\eta^{2}v_{L, n}^{2(\beta-1)} \vert \nabla_{A_{\varepsilon_{n}}(\cdot+ \tilde{y}_n)}v_{n}\vert^{2}
	+ \operatorname{Re}(\nabla v_{n} \overline{v_{n}}) \nabla (\eta^2 v_{L, n}^{2(\beta-1)})\\
	&=
	\eta^{2}v_{L, n}^{2(\beta-1)} \vert \nabla_{A_{\varepsilon_{n}}(\cdot+ \tilde{y}_n)}v_{n}\vert^{2}
	+ \vert v_{n}\vert \nabla \vert v_{n}\vert \nabla (\eta^2 v_{L, n}^{2(\beta-1)})\\
	&\geq \eta^{2}v_{L, n}^{2(\beta-1)} \vert \nabla \vert v_{n}\vert \vert^{2}+2\eta \nabla \eta v_{L, n}^{2(\beta-1)}\vert v_{n}\vert \nabla \vert v_n\vert,
	\end{align*}
	using also the fact that $u_{n}$ is a solution of problem \eqref{1.8} for $\varepsilon=\varepsilon_n$, the Young inequality (with $\tau>0$), (\ref{g3}), (\ref{1.4}), for $\alpha>4\pi$ and for a fixed $q>2$, given $0<\zeta<V_{0}$, there exists $C>0$ such that
	\begin{equation}\label{5.2}
	\begin{split}
	\int_{\mathbb{R}^{2}}\vert \nabla\vert v_{n}\vert\vert^{2}\eta^{2}v_{L, n}^{2(\beta-1)} dx
	&\leq
	\int_{\mathbb{R}^{2}}\vert \nabla\vert v_{n}\vert\vert^{2}\eta^{2}v_{L, n}^{2(\beta-1)}dx
	+2 \int_{\mathbb{R}^{2}} \eta \nabla \eta v_{L, n}^{2(\beta-1)}\vert v_{n}\vert \nabla \vert v_n\vert dx\\
	&\qquad
	+\int_{\mathbb{R}^{2}}V(\varepsilon_{n} x+\varepsilon_{n}\tilde{y}_{n})\eta^{2}v_{L, n}^{2(\beta-1)}\vert v_{n}\vert^{2}dx\\
	&\qquad
	+2\int_{\mathbb{R}^{2}}  \eta \vert\nabla \eta\vert v_{L, n}^{2(\beta-1)}\vert v_{n}\vert \vert\nabla \vert v_n\vert\vert
	-\zeta\int_{\mathbb{R}^{2}}\eta^{2}v_{L, n}^{2(\beta-1)} \vert v_{n}\vert^{2} dx
	\\
	&\leq
	\operatorname{Re}\int_{\mathbb{R}^{2}}(\nabla_{A_{\varepsilon_{n}}(\cdot+ \tilde{y}_n)} v_{n} \cdot\overline{\nabla_{A_{\varepsilon_{n}}(\cdot+ \tilde{y}_n)} z_{L, n}}) dx
	\\
	&\qquad
	+\operatorname{Re}\int_{\mathbb{R}^{2}}V(\varepsilon_{n} x+\varepsilon_{n}\tilde{y}_{n})v_{n}\overline{z_{L, n}}dx\\
	&\qquad
	+\tau\int_{\mathbb{R}^{2}}\vert \nabla\vert v_{n}\vert\vert^{2}\eta^{2}v_{L, n}^{2(\beta-1)}dx
	+\frac{1}{\tau}\int_{\mathbb{R}^{2}}  \vert\nabla \eta\vert^{2} v_{L, n}^{2(\beta-1)}\vert v_{n}\vert^{2} dx
	\\
	&\qquad
	-\zeta\int_{\mathbb{R}^{2}}\eta^{2}v_{L, n}^{2(\beta-1)} \vert v_{n}\vert^{2} dx\\
	&=
	\int_{\mathbb{R}^{2}}g({\varepsilon_{n} x+\varepsilon_{n}\tilde{y}_{n}}, \vert v_{n}\vert^{2})\eta^{2}v_{L, n}^{2(\beta-1)}\vert v_{n}\vert^{2} dx\\
	&\qquad
	+\tau\int_{\mathbb{R}^{2}}\vert \nabla\vert v_{n}\vert\vert^{2}\eta^{2}v_{L, n}^{2(\beta-1)}dx
	+\frac{4}{\tau r^2}\int_{R-r\leq \vert x\vert\leq R}  v_{L, n}^{2(\beta-1)}\vert v_{n}\vert^{2} dx
	\\
	&\qquad
	-\zeta\int_{\mathbb{R}^{2}}\eta^{2}v_{L, n}^{2(\beta-1)} \vert v_{n}\vert^{2} dx\\
	&\leq
	C\int_{\mathbb{R}^{2}}\vert v_{n}\vert^{q}(e^{\alpha\vert v_{n}\vert^{2} }-1) \eta^{2}v_{L, n}^{2(\beta-1)} dx\\
	&\qquad
	+\tau\int_{\mathbb{R}^{2}}\vert \nabla\vert v_{n}\vert\vert^{2}\eta^{2}v_{L, n}^{2(\beta-1)}dx
	+\frac{4}{\tau r^2}\int_{R-r\leq \vert x\vert\leq R}  \vert v_{n}\vert^{2\beta} dx.
	\end{split}
	\end{equation}
	Hence, choosing $\tau>0$ sufficiently small, we get
    \begin{equation}\label{5.4}
	\int_{\mathbb{R}^{2}}\vert \nabla\vert v_{n}\vert\vert^{2}\eta^{2}v_{L, n}^{2(\beta-1)}
	\leq C
	\Big[\int_{\vert x\vert\geq R-r}\vert v_{n}\vert^{q+2(\beta-1)}(e^{\alpha\vert v_{n}\vert^{2} }-1) dx
	+\frac{1}{r^2}\int_{R-r\leq \vert x\vert\leq R}  \vert v_{n}\vert^{2\beta} dx\Big].
	\end{equation}
	Moreover, arguing similarly to \eqref{5.2}, we can conclude that
	\begin{equation}\label{5.6}
	\int_{\mathbb{R}^{2}}\eta^{2} v_{L, n}^{2(\beta-1)}\vert v_{n}\vert^{2}dx
	\leq
	C\Big[\int_{\vert x\vert\geq R-r}\vert v_{n}\vert^{q+2(\beta-1)}(e^{\alpha\vert v_{n}\vert^{2} }-1) dx
	+\frac{1}{r^2}\int_{R-r\leq \vert x\vert\leq R}  \vert v_{n}\vert^{2\beta} dx\Big].
	\end{equation}
	On the other hand, using the Sobolev embedding, \eqref{5.4}, \eqref{5.6}, the H\"{o}lder inequality with $t,\sigma,\tau>1$, $1/\sigma+1/\tau=1/t$, $\sigma(q-2)\geq 2$, and \eqref{ineqe}, we have
 	\begin{equation}\label{5.8}
 	\begin{split}
 	\Vert w_{L, n}\Vert_{q}^{2}
 	& \leq
 	C \int_{\mathbb{R}^{2}}(\vert \nabla  w_{L, n}\vert^{2}+ \vert  w_{L, n}\vert^{2})dx\\
 	&\leq
 	C\Big(\int_{\mathbb{R}^{2}}\vert\nabla\eta\vert^{2} \vert v_{n}\vert^{2\beta}dx
 	+\beta^{2}\int_{\mathbb{R}^{2}}\eta^{2} v_{L, n}^{2(\beta-1)}\vert \nabla \vert v_{n}\vert\vert^{2}dx+ \int_{\mathbb{R}^{2}}\eta^{2} v_{L, n}^{2(\beta-1)}\vert v_{n}\vert^{2}dx\Big)\\
 	&\leq
 	C\beta^{2}\Big(\frac{1}{r^2}\int_{R-r\leq \vert x\vert\leq R} \vert v_{n}\vert^{2\beta}dx
 	+ 	\int_{\vert x\vert\geq R-r}\vert v_{n}\vert^{q+2(\beta-1)}(e^{\alpha\vert v_{n}\vert^{2} }-1) dx \Big)\\
 	&\leq
 	C\beta^{2}\left[\frac{R^{2/t}}{r^{2}}
 	+\Big(\int_{\vert x\vert\geq R-r}\vert v_{n}\vert^{\sigma(q-2)}dx\Big)^{1/\sigma}
 	\Big(\int_{\mathbb{R}^2}(e^{\tau \alpha\vert v_{n}\vert^{2} }-1)  dx\Big)^{1/\tau}\right]\\
 	&\qquad\qquad
 	\Big(\int_{\vert x\vert\geq R-r}\vert v_{n}\vert^{2\beta t/(t-1)}dx\Big)^{(t-1)/t}.
 	\end{split}
 	\end{equation}
Since $(\vert v_{n}\vert)$ is convergent in $H^{1}(\mathbb{R}^{2}, \mathbb{R})$, there exists $h\in H^{1}(\mathbb{R}^{2}, \mathbb{R})$   such that, for all $n\in \mathbb{N}$, $\vert v_{n}(x)\vert \leq h(x)$ a.e. in $\mathbb{R}^{2}$.
So, using Lemma \ref{le24}, for all $\tau>1$ and $\alpha>4\pi$, we know that
\begin{align}\label{5.9}
\int_{\mathbb{R}^{2}}(e^{\tau\alpha\vert v_{n}\vert^{2} }-1)dx\leq \int_{\mathbb{R}^{2}}(e^{\tau\alpha h^{2} }-1)dx<+\infty.
\end{align}
By \eqref{5.8} and \eqref{5.9}, it follows that
\begin{align*}
\Big(\int_{\vert x\vert\geq R}v_{L, n}^{q\beta }dx\Big)^{2/q}
\leq \Vert w_{L, n}\Vert_{q}^{2}\leq
C\beta^{2}\Big(1+\frac{R^{2/t}}{r^{2}}\Big)\Big(\int_{\vert x\vert\geq R-r}\vert v_{n}\vert^{2\beta t/(t-1)}\Big)^{(t-1)/t}
\end{align*}
and, applying the  Fatou's Lemma as $L\to+\infty$, we obtain
\begin{equation*}
\Big(\int_{\vert x\vert\geq R}|v_{n}|^{q\beta }dx\Big)^{2/q}
\leq  C\beta^{2}\Big(1+\frac{R^{2/t}}{r^{2}}\Big)\Big(\int_{\vert x\vert\geq R-r}\vert v_{n}\vert^{2\beta t/(t-1)}\Big)^{(t-1)/t}.
\end{equation*}
Arguing as in \cite{rL}, if we take $\zeta:=\frac{q(t-1)}{2t}$, $\beta:=\zeta^{m}$, with $m\in\mathbb{N}^*$, and $s:=\frac{2t}{t-1}$, we obtain
\begin{align*}
\Big(\int_{\vert x\vert\geq R}|v_{n}|^{s\zeta^{m+1} }dx\Big)^{1/(s\zeta^{m+1})}
\leq
C^{\zeta^{-m}} \zeta^{m\zeta^{-m}}
\Big(1+\frac{R^{2/t}}{r^{2}}\Big)^{1/(2\zeta^{m})}
\Big(\int_{\vert x\vert\geq R-r}\vert v_{n}\vert^{s\zeta^m}\Big)^{1/(s\zeta^{m})}
\end{align*}
for every $m\in\mathbb{N}^*$.
Then, for $r=r_{m}:=R/2^m$,  $m\in\mathbb{N}^*$, using also that $2/t<2$, we get
\begin{align*}
\Big(\int_{\vert x\vert\geq R}|v_{n}|^{s\zeta^{m+1} }dx\Big)^{1/(s\zeta^{m+1})}
&\leq
\Big(\int_{\vert x\vert\geq R-r_{m+1}}|v_{n}|^{s\zeta^{m+1} }dx\Big)^{1/(s\zeta^{m+1})}\\
&\leq
C^{\sum_{i=1}^{m}\zeta^{-i}}\zeta^{\sum_{i=1}^{m} i\zeta^{-i}}\exp \Big( \sum_{i=1}^{m}\frac{\ln(1+2^{2(i+1)})}{2\zeta^{i}}\Big)
\Big(\int_{\vert x\vert\geq R/2}|v_{n}|^{s\zeta }dx\Big)^{1/(s\zeta)}.
\end{align*}
Hence, passing to the limit as $m\rightarrow+\infty$ in the last inequality, we obtain
\begin{align}\label{5.10}
\Vert v_{n}\Vert_{L^{\infty}(B_R^c(0))}
\leq C \Big(\int_{\vert x\vert\geq R}|v_{n}|^{q }dx\Big)^{1/q}.
\end{align}
For $x_{0}\in \mathbb{R}^{2}$, we can use the same argument taking $\eta\in C_{0}^{\infty}(\mathbb{R}^{2}, [0, 1])$ with $\eta(x)=1$ if $\vert x-x_{0}\vert\leq \tilde{\rho}$, $\eta(x)=0$ if $\vert x-x_{0}\vert>2\rho$, with $\tilde{\rho} < \rho$, and $\vert \nabla \eta\vert\leq 2/\tilde{\rho}$, to prove that
\begin{align}\label{5.11}
\Vert v_{n} \Vert_{L^{\infty}(\overline{B_{2\rho}(x_0)})}\leq C
\Big(\int_{\vert x\vert\leq 2\rho}|v_{n}|^{q }dx\Big)^{1/q}.
\end{align}
Thus, by \eqref{5.10}, \eqref{5.11}, and using a standard covering argument and  the boundedness of $(|v_n|)$ in $L^q(\mathbb{R}^2,\mathbb{R})$, it follows that
\begin{align*}
\Vert v_{n} \Vert_{\infty}\leq C.
\end{align*}
Now, we use again the  convergence  of $(\vert v_{n}\vert)$  in $H^{1}(\mathbb{R}^{2}, \mathbb{R})$  on the right side
of \eqref{5.10} to get
\begin{align*}
\lim_{\vert x\vert\rightarrow+\infty}\vert v_{n}\vert=0\quad \text{uniformly in } n\in \mathbb{N}
\end{align*}
and the proof is complete.
\end{proof}

Now, we are ready to give the proof of Theorem \ref{mt}.

\begin{proof} [Proof of Theorem \ref{mt}]
Let $\delta>0$ be such that $M_{\delta}\subset\Lambda$. We want to show that there exists $\hat{\varepsilon}_{\delta}>0$ such that for any $\varepsilon\in (0, \hat{\varepsilon}_{\delta})$ and any  $u_\varepsilon \in \tilde{\mathcal{N}}_{\varepsilon}$ solution of problem \eqref{1.8}, it holds
\begin{align}\label{5.12}
\Vert u_\varepsilon \Vert^{2}_{L^{\infty}( \Lambda_{\varepsilon}^c)}\leq t_{a}.
\end{align}
We argue by contradiction and assume that there is a sequence $\varepsilon_{n}\rightarrow 0$ such that for every $n$ there exists $u_{n}\in \tilde{\mathcal{N}}_{\varepsilon_{n}}$ which satisfies $J'_{\varepsilon_{n}}(u_{n})=0$ and
\begin{align}\label{5.13}
\Vert u_{n} \Vert^{2}_{L^{\infty}( \Lambda_{\varepsilon_{n}}^c)}> t_{a}.
\end{align}
As in Lemma \ref{le51}, we have that $J_{\varepsilon_{n}}(u_{n})\rightarrow c_{V_{0}}$, and therefore we can use Proposition \ref{prop41} to obtain a sequence
 $(\tilde{y}_{n})\subset \mathbb{R}^{2}$ such that $y_n:=\varepsilon_{n}\tilde{y}_{n}\rightarrow y_{0}$
for some $y_{0}\in M$. Then, we can find $r>0$, such that $B_{r}(y_{n})\subset \Lambda$, and so $B_{r/\varepsilon_{n}}(\tilde{y}_{n})\subset \Lambda_{\varepsilon_{n}}$ for all $n$ large enough.\\
Using Lemma \ref{le51}, there exists $R>0$ such that $|v_{n}|^{2}\leq t_{a}$ in $B_R^c(0)$ and $n$ large enough, where $v_{n}=u_{n}(\cdot+\tilde{y}_{n})$. Hence $\vert u_{n}\vert^{2}\leq t_{a}$ in $B_{R}^c(\tilde{y}_{n})$ and $n$ large enough.
Moreover, if $n$ is so large that $r/\varepsilon_{n} >R$, then $ \Lambda_{\varepsilon_{n}}^c\subset  B_{r/\varepsilon_{n}}^c(\tilde{y}_{n})\subset  B_{R}^c(\tilde{y}_{n})$,
which gives $\vert u_{n}\vert^{2}\leq t_{a}$ for any $x\in \Lambda^c_{\varepsilon_{n}}$. This contradicts \eqref{5.13}
 and proves the claim.\\
Let now $\varepsilon_{\delta}:=\min\{\hat{\varepsilon}_{\delta}, \tilde{\varepsilon}_{\delta}\}$, where $\tilde{\varepsilon}_{\delta}>0$ is given by Theorem \ref{th41}. Then we have $\text{cat}_{M_{\delta}}(M)$ nontrivial solutions to problem \eqref{1.8}.
If $u_\varepsilon\in\tilde{\mathcal{N}}_{\varepsilon}$ is one of these solutions, then, by  \eqref{5.12} and the definition of $g$, we conclude that $u_\varepsilon$ is also a solution to problem \eqref{1.6}. \\
Finally, we study
the behavior of the maximum points of $\vert\hat{u}_{\varepsilon}\vert$, where $\hat{u}_{\varepsilon}(x):=u_{\varepsilon}(x/\varepsilon)$ is a solution to problem \eqref{1.1}, as $\varepsilon\to 0^+$.\\
Take $\varepsilon_{n}\rightarrow 0^+$ and the sequence $(u_n)$ where each $u_n$ is a solution of \eqref{1.8} for $\varepsilon=\varepsilon_n$. In view of (\ref{g2}), there exists $\gamma\in (0, t_{a})$ such that
\begin{equation*}
g(\varepsilon x, t^{2})t^{2}\leq \frac{V_{0}}{2}t^{2},\quad\text{for all}\,\, x\in \mathbb{R}^{2},\, \vert t\vert\leq\gamma.
\end{equation*}
Arguiguing as above
 we can take $R>0$ such that, for $n$ large enough,
\begin{equation}\label{5.14}
\Vert  u_n\Vert_{L^{\infty}(B_{R}^{c}(\tilde{y}_{n}))}<\gamma.
\end{equation}
Up to a subsequence, we may also assume that for $n$ large enough
\begin{equation}\label{5.15}
\Vert u_n \Vert_{L^{\infty}(B_{R}(\tilde{y}_{n}))}\geq\gamma.
\end{equation}
Indeed, if \eqref{5.15} does not hold, up to a subsequence, if necessary,
we have $\Vert u_n \Vert_{\infty}<\gamma$. Thus, since $J'_{\varepsilon_{n}}(u_{\varepsilon_{n}})=0$, using (\ref{g5}) and the diamagnetic inequality \eqref{2.1} that
\begin{align*}
\int_{\mathbb{R}^{2}}(\vert\nabla  \vert u_n\vert\vert^{2}+ V_{0}\vert u_n\vert^{2})dx
\leq \int_{\mathbb{R}^{2}}g(\varepsilon_{n} x, \vert u_n\vert^{2})\vert u_n\vert^{2}dx
\leq \frac{V_{0}}{k}\int_{\mathbb{R}^{2}}\vert u_n\vert^{2}dx
\end{align*}
and, being $k>1$, $\Vert  u_n \Vert=0$, which is a contradiction.\\
Taking into account \eqref{5.14} and \eqref{5.15}, we can infer that the global maximum points $p_{n}$ of $\vert u_{\varepsilon_{n}}\vert$ belongs to $B_{R}(\tilde{y}_{n})$, that is
$p_{n}=q_{n}+\tilde{y}_{n}$ for some $q_{n}\in B_{R}$. Recalling that the associated solution of problem \eqref{1.1} is $\hat{u}_{n}(x)=u_n(x/\varepsilon_{n})$, we can see that
a  maximum point $\eta_{\varepsilon_{n}}$ of $\vert \hat{u}_{n}\vert$ is $\eta_{\varepsilon_{n}}=\varepsilon_{n}\tilde{y}_{n}+\varepsilon_{n}q_{n}$. Since $q_{n}\in B_{R}$,
$\varepsilon_{n}\tilde{y}_{n}\rightarrow y_{0}$ and $V(y_{0})=V_{0}$, the continuity of $V$ allows to conclude that
\begin{equation*}
\lim_{n}V(\eta_{\varepsilon_{n}})=V_{0}.
\end{equation*}
\end{proof}

\vspace{0.5 cm}

\section*{Acknowledgements} The authors would like to thank the anonymous referees for their valuable suggestions and comments.

\end{document}